\documentclass[12pt]{article}
\usepackage{graphicx}
\usepackage{amsmath}
\usepackage{amsfonts}
\usepackage{mathtools}
\usepackage{amsthm}
\usepackage{dsfont}
\usepackage{esint}
\usepackage[T1]{fontenc}
\usepackage{url}
\usepackage{color}
\usepackage[margin=1in]{geometry}
\usepackage{amssymb,bm}
\usepackage{amsmath}
\usepackage{amsthm}
\usepackage{graphicx}
\usepackage[active]{srcltx} 
\usepackage{hyperref}
\hypersetup{pdfborder=0 0 0}

\newtheorem{theorem}{{\sc Theorem}}[section]

\newtheorem{lemma}[theorem]{{\sc Lemma}}

\newtheorem{remark}[theorem]{Remark}

\newtheorem{definition}[theorem]{Definition}

\def\XXint#1#2#3{{\setbox0=\hbox{$#1{#2#3}{\int}$ }
\vcenter{\hbox{$#2#3$ }}\kern-.6\wd0}}

\newcommand{\Ga}{\alpha}

\newcommand{\Gd}{\delta}
\newcommand{\Ge}{\epsilon}

\newcommand{\Gg}{\gamma}

\newcommand{\Gl}{\lambda}

\newcommand{\Gth}{\theta}

\newcommand{\Gs}{\sigma}

\newcommand{\GO}{\Omega}

\bmdefine\BGa{\alpha}
\bmdefine\BGb{\beta}
\bmdefine\BGd{\delta}
\bmdefine\BGe{\epsilon}
\bmdefine\BGve{\varepsilon}
\bmdefine\BGf{\phi}
\bmdefine\BGvf{\varphi}
\bmdefine\BGg{\gamma}
\bmdefine\BGc{\chi}
\bmdefine\BGi{\iota}
\bmdefine\BGk{\kappa}
\bmdefine\BGl{\lambda}
\bmdefine\BGn{\eta}
\bmdefine\BGm{\mu}
\bmdefine\BGv{\nu}
\bmdefine\BGp{\pi}
\bmdefine\BGth{\theta}
\bmdefine\BGvth{\vartheta}
\bmdefine\BGr{\rho}
\bmdefine\BGvr{\varrho}
\bmdefine\BGs{\sigma}
\bmdefine\BGvs{\varsigma}
\bmdefine\BGt{\tau}
\bmdefine\BGj{\tau}
\bmdefine\BGu{\upsilon}
\bmdefine\BGo{\omega}
\bmdefine\BGx{\xi}
\bmdefine\BGy{\psi}
\bmdefine\BGz{\zeta}
\bmdefine\BGD{\Delta}
\bmdefine\BGF{\Phi}
\bmdefine\BGG{\Gamma}
\bmdefine\BGL{\Lambda}
\bmdefine\BGP{\Pi}
\bmdefine\BGT{\Theta}
\bmdefine\BGS{\Sigma}
\bmdefine\BGU{\Upsilon}
\bmdefine\BGO{\Omega}
\bmdefine\BGX{\Xi}
\bmdefine\BGY{\Psi}

\bmdefine\BCA{{\mathcal A}}
\bmdefine\BCB{{\mathcal B}}
\bmdefine\BCC{{\mathcal C}}
\bmdefine\BCD{{\mathcal D}}
\bmdefine\BCE{{\mathcal E}}
\bmdefine\BCF{{\mathcal F}}
\bmdefine\BCG{{\mathcal G}}
\bmdefine\BCH{{\mathcal H}}
\bmdefine\BCI{{\mathcal I}}
\bmdefine\BCJ{{\mathcal J}}
\bmdefine\BCK{{\mathcal K}}
\bmdefine\BCL{{\mathcal L}}
\bmdefine\BCM{{\mathcal M}}
\bmdefine\BCN{{\mathcal N}}
\bmdefine\BCO{{\mathcal O}}
\bmdefine\BCP{{\mathcal P}}
\bmdefine\BCQ{{\mathcal Q}}
\bmdefine\BCR{{\mathcal R}}
\bmdefine\BCS{{\mathcal S}}
\bmdefine\BCT{{\mathcal T}}
\bmdefine\BCU{{\mathcal U}}
\bmdefine\BCV{{\mathcal V}}
\bmdefine\BCW{{\mathcal W}}
\bmdefine\BCX{{\mathcal X}}
\bmdefine\BCY{{\mathcal Y}}
\bmdefine\BCZ{{\mathcal Z}}

\bmdefine\Bzr{ 0}
\bmdefine\Ba{ a}
\bmdefine\Bb{ b}
\bmdefine\Bc{ c}
\bmdefine\Bd{ d}
\bmdefine\Be{ e}
\bmdefine\Bf{ f}
\bmdefine\Bg{ g}
\bmdefine\Bh{ h}
\bmdefine\Bi{ i}
\bmdefine\Bj{ j}
\bmdefine\Bk{ k}
\bmdefine\Bl{ l}
\bmdefine\Bm{ m}
\bmdefine\Bn{ n}
\bmdefine\Bo{ o}
\bmdefine\Bp{ p}
\bmdefine\Bq{ q}
\bmdefine\Br{ r}
\bmdefine\Bs{ s}
\bmdefine\Bt{ t}
\bmdefine\Bu{ u}
\bmdefine\Bv{ v}
\bmdefine\Bw{ w}
\bmdefine\Bx{ x}
\bmdefine\By{ y}
\bmdefine\Bz{ z}
\bmdefine\BA{ A}
\bmdefine\BB{ B}
\bmdefine\BC{ C}
\bmdefine\BD{ D}
\bmdefine\BE{ E}
\bmdefine\BF{ F}
\bmdefine\BG{ G}
\bmdefine\BH{ H}
\bmdefine\BI{ I}
\bmdefine\BJ{ J}
\bmdefine\BK{ K}
\bmdefine\BL{ L}
\bmdefine\BM{ M}
\bmdefine\BN{ N}
\bmdefine\BO{ O}
\bmdefine\BP{ P}
\bmdefine\BQ{ Q}
\bmdefine\BR{ R}
\bmdefine\BS{ S}
\bmdefine\BT{ T}
\bmdefine\BU{ U}
\bmdefine\BV{ V}
\bmdefine\BW{ W}
\bmdefine\BX{ X}
\bmdefine\BY{ Y}
\bmdefine\BZ{ Z}

\begin{document}
\title{Singularity removal in nonlocal operator inequalities on domains and applications: A $((Q,1)\implies (\GO,\rho))$ theorem }
\author{Davit Harutyunyan\thanks{University of California Santa Barbara, harutyunyan@math.ucsb.edu}}
\date{}

\maketitle

\begin{abstract}

We prove that certain classes of nonlocal integral inequalities on bounded uniform domains, involving operators, may be reduced to the local analogues on cubes without the singular weight. The theory applies to two-sided inverse doubling and quasi-monotone kernels with some moment conditions, and with a left-hand-side operator with average bi-Lipschitz properties. Direct applications to fractional semi-norm rigidity inequalities in nonlocal continuum mechanics, such as fractional Korn's Inequalities, Geometric Rigidity Estimates, and rigidity of some types of affine deformations, that are common in mechanics, are presented.

\end{abstract}

\section{Introduction}
\label{sec:1}

Let $n,N,m_1,m_2\in\mathbb N$ and $p\in [1,\infty).$ Let the kernel $\rho$ and the operators $\BA$ and $\BB$ be defined as follows: Given a bounded domain $\GO\subset\mathbb R^n,$ let the kernel
\begin{equation}
\label{1.1}
\rho(x,y)\colon (\GO\times \GO)\setminus \{x=y\} \longrightarrow (0,\infty),
\end{equation} 
be positive, and let $\mathcal{F}\subset\mathcal{P}(\GO)$ be the collection of all connected measurable subsets of $\GO.$ 
Assume the operator$\BA(E,f)$ is defined as
\begin{equation}
\label{1.2}
\BA\colon \mathcal{F}\times L^1_{loc}(\GO,\mathbb R^N)\longrightarrow L^1_{loc}(\GO,\mathbb R^{m_1}),
\end{equation}
and the operator $\BB(f)$ is defined as
\begin{equation}
\label{1.3}
\BB\colon L^1_{loc}(\GO,\mathbb R^N)\longrightarrow L^1_{loc}(\GO \times \GO,\mathbb R^{m_2}).
\end{equation}
It is worth emphasizing, that the domain $\GO\subset\mathbb R^n$ is not a unique fixed domain in the above setting. \\

\textit{The objective of this manuscript is finding sufficient conditions under which there exists a 
non-degenerate cube $Q\subseteq \GO,$ 
such that the inequality
\begin{equation}
\label{1.4}
\int_{\GO}\int_{\GO}\rho(x,y)|\BA(Q,f)(x)-\BA(Q,f)(y)|^pdxdy\leq C\int_{\GO}\int_{\GO}\rho(x,y)|\BB f(x,y)|^pdxdy
\end{equation}
holds for all $f\in L^1_{loc}(\GO,\mathbb R^N),$ where the constant $C$ may depend on 
$n,N,m_1,m_2,p,\BA,\BB,\rho,$ and $\GO$.}\\
 
The setting (\ref{1.4}) is inspired by the so-called fractional Geometric Rigidity Estimate (abbreviated as FGRE) and the fractional Korn's first inequality (abbreviated as FKI), however, obviously, the scope of (\ref{1.4}) is much wider and includes more general operators than the ones generated by those inequalities. The more general operator estimate (\ref{1.4}) puts those rigidity estimates under one umbrella, as will be clear below, after formulating the two statements mentioned above. We proceed to the formulation of the most general form of the fractional Korn inequality known in the literature, which reads as follows:\\

\noindent \textbf{The FKI.} \textit{Let $\GO\subset\mathbb R^n$ be a bounded uniform domain\footnote{Note, that unifomity automatically implies openness and connectedness} and let $s\in (0,1),$ $1<p<\infty$ and $n\in\mathbb N,$ with $n\geq 2.$ 
There exists a constant $C,$ depending only on $n,s,p$ and $\GO$ such that for all vector fields 
$\Bu\in W^{s,p}(\Omega,\mathbb R^n)$ one has
\begin{equation}
\label{1.5}
 \inf_{\Ba\in \mathfrak{se}(n)}[\Bu-\Ba]_{W^{s,p}(\Omega)}^p\leq C[\Bu]_{\mathcal{X}^{s,p}(\Omega)}^p,
\end{equation}
where 
\begin{equation}
\label{1.6}
\mathfrak{se}(n)=\{\Ba(x)=\BA x+b\colon\mathbb R^n\to\mathbb R^n\}
\quad\text{with}\quad \BA\in\mathbb R^{n\times n},\ \BA+\BA^T=0,\ b\in\mathbb R^n,
\end{equation}
is the space of infinitesimal rigid motions of $\mathbb R^n.$ Each expression in (\ref{1.5}) is defined in the next lines below.}\\

Recall, that for $s\in (0,1),$ and for a vector field $\Bu\in L^p(\Omega,\mathbb R^n),$ the Gagliardo seminorm associated to the fractional $s-$derivative of $\Bu$ is given by 
\begin{equation}
\label{1.7}
[\Bu]_{W^{s,p}(\Omega)}^p=\int_{\Omega}\int_{\Omega}\frac{|\Bu(x)-\Bu(y)|^p}{|x-y|^{n+ps}}dx dy, 
\end{equation}
while the quantity 
$\|\Bu\|_{W^{s,p}(\Omega)}=[\Bu]_{W^{s,p}(\Omega)}+\|\Bu\|_{L^p(\Omega)}$ defines a norm 
[\ref{bib:Bou.Bre.Mir.1},\ref{bib:Maz.Sha.},\ref{bib:DiN.Pat.Val.}]. The nonlocal analogue of the linear elastic energy is the right-hand side of (\ref{1.5}) and is given by (see [\ref{bib:Mengesha.1}]): 
\begin{equation}
\label{1.8}
[\Bu]_{\mathcal{X}^{s,p}(\Omega)}^p=\int_{\Omega}\int_{\Omega}\frac{|(\Bu(x)-\Bu(y))\cdot(x-y)|^p}{|x-y|^{n+ps+p}}dxdy.
\end{equation}
The expression in (\ref{1.8}) has arisen in the theory of linear Peridynamics [\ref{bib:Silling.1},\ref{bib:Silling.2},\ref{bib:Sil.Ept.Wec.Xu.Ask.},\ref{bib:Du.Zho.}] as the energy in the small strain regime in the analogy of linear elasticity, where the symmetric gradient $e(\Bu)$ is the linear strain, and the integral $\int_\Omega |e(\Bu)|^p dx$ is a replacement of the linear elastic energy. The estimate in the above formulation (\ref{1.5}) was recently proven by Acosta and coauthors, see [\ref{bib:Aco.Dre.Dur.Lop.Oje.}, Theorem~6.1]. The story behind Korn's inequality (\ref{1.5}) goes back to Mengesha's paper [\ref{bib:Mengesha.1}], where it was first proven for the half-space in $L^2$ and for compactly supported fields and under the  condition $s\neq\frac{1}{2}.$ For later developments of (\ref{1.5}) see [\ref{bib:Men.Sco.1},\ref{bib:Men.Sco.2},\ref{bib:Rutkowski.2},\ref{bib:Har.Mik.1}] and [\ref{bib:Har.Men.Mik.Sco.}], where it was established for open, bounded, connected $C^1$ domains, before the uniform domain proof in [\ref{bib:Aco.Dre.Dur.Lop.Oje.}].  

Coming back to the FGRI, it is easy to prove that the measurable exact a.e. isometries (rigid motions) of $\mathbb R^n$ are of 
the form $\Br(x)=\BR x+b$, where $\BR\in \mathrm{O}(n)$ is an orthogonal matrix. Consequently, recall the standard notation
\begin{equation}
\label{1.9}
\mathfrak{so}(n)=\{\Br(x)=\BR x+b\colon\mathbb R^n\to\mathbb R^n\}
\qquad\text{with}\qquad \BR\in\mathrm{O}(n),\ b\in\mathbb R^n.
\end{equation} 
of the space of exact isometries of $\mathbb R^n.$ The FGRE would be the nonlinear analogue of (\ref{1.5}) with $\Br$ instead 
of $\Ba,$ the space $\mathfrak{so}(n)$ instead of $\mathfrak{se}(n),$ and a nonlinear elastic energy (yet to be defined) on the right-hand side instead of the linear counterpart. The nonlinear fractional elastic energy has to vanish on orientation preserving isometries of $\mathbb R^n,$ i.e., on rigid motions, because those carry zero elastic energy.  However, because in general there is no notion of orientation preservation for $L^1_{loc}$ fields, it is natural co consider $O(n)$ instead of $SO(n)$ in the nonlocal case. A natural energy density candidate that may appear in the nonlinear nonlocal elastic energy is the quantity 
$$
||\Bu(x)-\Bu(y)|-|x-y||,
$$
 which vanishes exactly on exact isometries, i.e., affine fields with a constant gradient in $O(n).$ The analogous estimate written for $SO(n)$ fails to hold. We refer the reader to [\ref{bib:Fri.Jam.Mue.}] for the classical local formulation (GRE) of the FGRE. The story behind the FGRI is worth telling. When we were at final stages of typing and polishing the current manuscript and the companion paper [\ref{bib:Har.Mik.2}], where we formulate and prove the FGRI, the FKI, and other nonlocal rigidity estimates in bounded uniform domains for the case $1\leq p<\infty$, e.g.,  [\ref{bib:Har.Mik.2}, Theorems~2.1-2.8], a formulation and proof of the FGRE was still open in public, and James Scott and Tadele Mengesha kindly shared pdf copies of their newly posted on arxiv.org papers [\ref{bib:Scott}] and [\ref{bib:Men.Sco.3}] on the FGRE and the FKI with us. The paper [\ref{bib:Scott}] establishes the formulation and proof of the FGRI in open bounded connected Lipschitz subsets of $\mathbb R^n$ and with general radial kernes $\rho(|x-y|)$ that covers the standard kernel $\rho(x,y)=|x-y|^{-n-ps}$ appearing in (\ref{1.7}), (\ref{1.8}), and (\ref{1.5}), see [\ref{bib:Scott}, Theorem~2.1], while the paper [\ref{bib:Men.Sco.3}] proves FKI on bounded Lipschitz domains also for the endpoint cases $p=1,\infty,$ see [\ref{bib:Men.Sco.3}, Theorems~1.1-1.3]. Our assumptions on the kernel $\rho(x,y)$ are slightly different and our theorem on the FGRI proven in [\ref{bib:Har.Mik.2}, Theorem~2.1] based on the localization technique of the current manuscript is different, and we establish it on uniform domains, which strictly contains the class of Lipschitz domains. A detailed comparison of [\ref{bib:Har.Mik.2}, Theorem~2.1] and [\ref{bib:Scott}, Theorem~2.1] appears in Section~1 of [\ref{bib:Har.Mik.2}], were we refer the interested readers to. It is not surprising, that our formulation of the FGRI coincides with that of Scotts', where both naturally follow standard principles of mathematics and physics. A general heuristic formulation of the FGRI, without exact conditions on the kernel $\rho$ and the 
domain $\GO$ is as follows:\\ 

\noindent \textbf{The FGRI.} \textit{Let $n\in\mathbb N$ with $n\geq 2,$ let $1\leq p<\infty,$ and let $\GO\subset\mathbb R^n$ 
be an open bounded sufficiently regular (explicitly specified as Lipschitz in [\ref{bib:Scott}] and uniform in [\ref{bib:Har.Mik.2}]) 
domain. Assume the kernel  $\rho(x,y)\colon \GO^2\to[0,\infty)$ satisfies certain monotonicity, growth and integrability conditions (explicitly specified in [\ref{bib:Scott}] and [\ref{bib:Har.Mik.2}]). Then, there exists a constant $C,$ depending only on $n,p,\rho$ 
and $\GO$ such that for all vector fields $\Bu\in L^p(\GO,\mathbb R^n)$ one has}
\begin{align}
\label{1.10}
 \inf_{\Br \in \mathfrak{so}(n)}\int_{\GO}\int_{\GO} \rho(x,y)&\huge|\Bu(x)-\Bu(y)-[\Br(x)-\Br(y)]\huge|^pdxdy\\ \nonumber
 &\leq C\int_{\GO}\int_{\GO} \rho(x,y)\left| |\Bu(x)-\Bu(y)|-|x-y|\right|^pdxdy. 
\end{align}
Note, that if one takes $\mathrm{SO}(n)$ instead of $\mathrm{O}(n),$ then (\ref{1.10}) would fail, due to the possibility of choosing 
an isometry of the form $\Br(x)=(\BI-2\Be_1\otimes \Be_1)\BR x,$ with $\BR\in \mathrm{SO}(n),$ which is an orientation altering orthogonal transformation, that is en exact isometry of $\mathbb R^n.$ \\

Coming back to the general operator formulation (\ref{1.4}), note that both the FKI and FGRI are particular cases of it, where for each inequality, one has to define the operators $\BA$ and $\BB,$ which is done in Section~4.\\

The necessary conditions we would have to impose on the kernel $\rho,$ the operators $\BA$ and $\BB$ and the domain $\GO$
occur in the proof schemes. It is necessary and natural to assume translational quasi-invariance on $\rho$ in the context of (\ref{1.4}), which is the existence of a fixed constant $c>0$ such that $\rho(x+z,y+z)\leq c\rho(x,y)$ for all $x,y,z\in\mathbb R^n.$ However, this seemingly more ambitious translational quasi-invariance is in fact equivalent to the standard translational invariance $\rho(x+z,y+z)=\rho(x,y)$ in the context of (\ref{1.4}) after replacing the kernel $\rho$ by an equivalent translation-invariant kernel 
$\rho'(x,y)$ by defining 
$$
\rho'(x,y)=\inf_{z\in\mathbb R^n}\rho(x+z,y+z).
$$ 
Indeed, the translational quasi-invariance implies $ \rho(x,y)=\rho(x+z+(-z),y+z+(-z))\leq c\rho(x+z,y+z),$ and thus the equivalence 
$\rho'(x,y)\simeq \rho(x,y).$ Consequently, adopting notational simplicity, we can simply work under the condition 
\begin{equation}
\label{1.11}
\rho(x,y)=\rho(x-y),\qquad\text{ for all} \qquad x,y\in\mathbb R^n\qquad \text{with} \qquad x\neq y,
\end{equation}
which we assume in what follows. Other conditions that we impose upon $\rho$ for the analysis to go through are the two-sided 
inverse doubling property, the quasi-decreasing property in $|x|$ and a moment condition. We assume the following explicit hypotheses on $\rho:$\\

\noindent\textbf{The hypotheses on $\rho$.} 
\begin{itemize}
\item[(H$\rho$1)](quasi-decreasing). There exist a constant $c_1>0$ such that for all $x,y\in\mathbb R^n$ with $|x|\leq |y|$ one has
\begin{equation}
\label{1.12}
\rho(y) \leq c_1\rho(x).
 \end{equation}
 \item[(H$\rho$2)](inverse doubling). There exist constants $0<c_2\leq c_2'<1$ such that for all  $x,y\in\mathbb R^n$ with $|x|=|y|\neq 0$
 one has
\begin{equation}
\label{1.13}
c_2\rho(x) \leq \rho(2y)\leq c_2' \rho(x).
 \end{equation}
 \end{itemize}
By notational abuse, define for $t> 0$
\begin{equation}
\label{1.14}
\rho(t)=\rho(te_1),
\end{equation}
which is an equivalent radial kernel. Then it is easy to see that we have $\rho(t)\colon (0,\infty)\to (0,\infty)$ by (H$\rho$1)-(H$\rho$2). Also, we get from (H$\rho$1)-(H$\rho$2) the one dimensional analogous conditions (H$\rho$1')-(H$\rho$2'). We impose the average 
moment condition (H$\rho$3) below in addition.
\begin{itemize}
\item[(H$\rho$1')](quasi-decreasing). For all $0<t_1\leq t_2<\infty,$ one has
\begin{equation}
\label{1.15}
\rho(t_2) \leq c_1\rho(t_1).
 \end{equation}
 \item[(H$\rho$2')](inverse doubling).  For all $t>0$ one has 
\begin{equation}
\label{1.16}
c_2\rho(t) \leq \rho(2t)\leq c_2' \rho(t).
 \end{equation}
\item[(H$\rho$3).] There exists a constant $c_3>0$ such that
\begin{equation}
\label{1.17}
\int_0^a t^{n+p-1}\rho(t)dt\leq c_3a^{n+p}\rho(a) \qquad\text{for all}\qquad a>0.
\end{equation}
\end{itemize}
In Section~2, in some of the theorems, we will assume two variants of additional moment conditions on $\rho,$ which we formulate below.
\begin{itemize}
\item[(H$\rho$4)] There exists a constant $c_4>0$ such that 
\begin{equation}
\label{1.18}
\int_a^1 t^{n-1}\rho(t)dt\leq c_4a^n\rho(a),\qquad\text{for all}\qquad 0<a<1.
\end{equation}
\item[(H$\rho$5)] There exists a constant $c_5>0$ such that 
\begin{equation}
\label{1.19}
\int_a^1 t^{n-1}\rho(t)dt\leq c_5(1+a^n\rho(a)),\qquad\text{for all}\qquad 0<a<1.
\end{equation}\end{itemize} 
In most applications, the kernel $\rho$ is singular with $\lim_{x\to 0}\rho(x)=\infty,$ otherwise one can simply unweight the 
inequality by discarding the weight from it. Note that conditions (H$\rho$1') and (H$\rho$2') together automatically imply that 
$\rho$ is singular at the origin, because for $t\in[2^{-k-1},2^{-k}),$ one has 
$$
\rho(t)\geq \frac{1}{c_1}\rho(2^{-k})\geq \frac{1}{c_1(c_2')^k}\rho(1)\to \infty \qquad \text{as} \qquad k\to\infty.
$$ 
Also, in most applications one has the case of a radial kernel $\rho(x,y)=\rho(|x-y|),$ in particular, in the standard kernel
\begin{equation}
\label{1.20}
\rho(x,y)=\frac{1}{|x-y|^{n+sp}},
\end{equation}
where $s\in (0,1)$ is the derivative measure, is radial. It is easy to verify that the standard kernel fulfills 
the hypotheses (H$\rho$1)-(H$\rho$4)  with 
$$
c_1=1,\qquad c_2=c_2'=2^{-n-ps},\qquad c_3=\frac{1}{p(1-s)},\qquad c_4=c_5=\frac{1}{ps}.
$$
Thus our results are applicable to Peridynamics, where one deals with the standard kernel. In the case when the density 
$\rho$ is singular, then (\ref{1.4}) becomes a nonlocal integral inequality. Those kind of estimates play a central role in the 
theory of Perydinamics, where our main motivation came from and where particular cases of inequality (\ref{1.4}) arise. 
However, (\ref{1.4}) can be applicable in fractional calculus in general. We will discuss applications in Section~5. Next, 
on the operator $\BA,$ we will impose an average bi-Lipschitz condition on cubes as follows. For any two bounded 
measurable sets $E,F\subseteq \GO,$ denote
 \begin{equation}
\label{1.21}
\BGd_{E,F}(x)=\BA(E, f)(x)-\BA(F,f)(x).
\end{equation}
We assume the following conditions on $\BA.$\\

\noindent\textbf{The hypotheses on $\BA$.} 

\begin{itemize}
\item[(HA1).] There exists a constant $C_1=C_1(\BA)>0$ such that for every cubes $P_1,P_2,P,Q\subseteq\GO$ such that 
$P_1\cap P_2\neq \emptyset$ one has
\begin{align}
\label{1.22}
\fint_{P_1} \fint_{P_2} |\BGd_{P,Q}(x)- & \BGd_{P,Q}(y)|^p dxdy\\ \nonumber
 &\leq C_1(\BA) \left(\frac{\mathrm{diam}(P_1\cup P_2)}{\mathrm{diam}(P\cup Q)}\right)^p
 \fint_{P} \fint_{Q} |\BGd_{P,Q}(x)-\BGd_{P,Q}(y)|^p dxdy
\end{align}
\item[(HA2).] There exists a constant $C_2=C_2(\BA)>0$ such that for every cubes $P_1,P,Q\subseteq\GO$ with 
$P \cap Q\neq \emptyset,$  one has
\begin{align}
\label{1.23}
\fint_{P} \fint_{Q} |\BGd_{P,Q}(x)- & \BGd_{P,Q}(y)|^p dxdy\\ \nonumber
& \leq C_2(\BA)\left(\frac{\mathrm{diam}(P\cup Q)}{\mathrm{diam} \ P_1}\right)^p
\fint_{P_1} \fint_{P_1} |\BGd_{P,Q}(x)-\BGd_{P,Q}(y)|^p dxdy
\end{align}
\end{itemize}

We impose no hypotheses on the operator $\BB$. \\

\noindent\textbf{The hypotheses on the domain $\GO$.}\\

 We will prove one of the main theorems (Theorem~2.2) of the paper in  uniform domains. These are certain type of John domains that strictly include Lipschitz domains. We recall the definitions of John and uniform domains below.
 
\begin{definition}[John domain] A bounded domain $\GO\subset\mathbb R^n$ is called a John domain, if there exists a constant 
$0<c\leq 1$ and a fixed point $x_0\in \GO$ such that for every other point $x\in\GO,$ there exists a rectifiable curve 
$\Gg\colon[0,s]\to \GO,$ parameterized by its arc length, that starts at $x,$ ends at $x_0,$ and satisfies 
$\mathrm{dist}(\Gg(t),\partial\GO)\geq ct$ for all $t\in [0,s],$ where $s$ is the length of $\Gg.$ 
\end{definition}
We refer the reader to the original paper of John [\ref{bib:John}] and the paper [\ref{bib:Mar.Sar.}] for more details on John domains. 
Next we recall the definition of uniform domains. 

\begin{definition}[Uniform domain] A domain $\GO\subset\mathbb R^n$ is called a uniform domain, if there exists a constant 
$C\geq 1,$ such that for every two points $x,y\in\GO,$ there exists a rectifiable curve $\Ga\colon[0,s]\to \GO,$ 
that starts at $x,$ ends at $y,$ and satisfies $\min(\ell(x,\Gg(t)),\ell(y,\Gg(t)))\leq C\mathrm{dist}(\Gg(t),\partial\GO)$ 
for all $t\in [0,s],$ where $\ell(z,\Gg(t))$ is the length of the curve between the points $z$ and $\Gg(t)$ for $z=x,y.$
\end{definition}
We refer the reader to [\ref{bib:Mar.Sar.},\ref{bib:Geh.Osg.},\ref{bib:Jones},\ref{bib:Jer.Ken.}] for more details on uniform domains. The class of Lipschitz domains is a proper subset of the class of uniform domains, while the later is a proper subset of the class of John domains. Also, John domains are open, bounded and simply connected by definition.

\section{Main results}
\setcounter{equation}{0}
\label{sec:2}

In what follows, $\GO\subset\mathbb R^n$ is an open, bounded, connected uniform domain. The below theorems contain the main results of the paper. The second and third theorem reduce the inequality to the case of nested concentric cubes, while the first theorem reduces the complexity of the problem further, reducing the estimate to an analogous unweighted estimate, making it local. 

\begin{theorem}[Singularity removal]
\label{Thm:2.1}
Let $n\geq 2,$ $1\leq p<\infty$ and take $\GO=P\subset\mathbb R^n$ to be a non-degenerate cube in (\ref{1.1})-(\ref{1.3}). Assume the kernel $\rho$ satisfy hypotheses (H$\rho$1')-(H$\rho$3) and the operator $\BA$ satisfy hypotheses (HA1)-(HA2). Then statement (i) implies statement (ii) below. 
\begin{itemize}
\item[(i)] There exists a constant $C_3=C_3(n,p,\BA,\BB)$ such that for every cube $Q\subseteq P,$ the unweighted inequality 
\begin{equation}
\label{2.1}
\int_{Q}|\BA(Q,f)(x)|^pdx\leq \frac{C_3}{|Q|}\int_{4Q}\int_{4Q}|\BB f(x,y)|^pdxdy.
\end{equation}
holds for all $f\in L^1_{loc}(P,\mathbb R^N).$ 
\item[(ii)] There exists a constant $\widetilde C_3 =\widetilde  C_3(n,p,\BA,\BB)$ such that for every cube $Q\subseteq P,$ the weighted inequality
\begin{equation}
\label{2.2}
\int_{Q}\int_{Q}\rho(x,y)|\BA(Q,f)(x)-\BA(Q,f)(y)|^pdxdy\leq \widetilde C_3(c_3^p+1)\int_{4Q}\int_{4Q}\rho(x,y)|\BB f(x,y)|^pdxdy
\end{equation}
holds for all $f\in L^1_{loc}(P,\mathbb R^N).$ The constant $c_3$ in (\ref{2.2}) is coming from (H$\rho$3).

\end{itemize}
\end{theorem}

\vspace{0.1cm}

\begin{remark}
\label{Rem.2.2}
A careful inspection reveals, that if one chooses the geometric ratio appropriately, instead of $2,$ then the proof of Theorem~2.1
 can be adjusted to every $\Gs>1$ geometric ratio and hence the proof of Theorem~2.1 holds for $\sigma-$enlarged cubes 
 ($\Gs Q$ instead of $4Q$) for every $\Gs>1.$ On the other hand, combining Theorems~2.1 and 2.3, one gets Theorem~2.1
  with $\Gs=1$ as well, i.e., without enlarging the cubes. 
\end{remark}

\vspace{0.1cm}

\begin{theorem}[Domain reduction]
\label{Thm:2.3}
Let $n\geq 2$ and $1\leq p<\infty,$ let $\GO\subset\mathbb R^n$ be a bounded uniform domain, and let the operator $\BA$ satisfy hypothesis (HA1)-(HA2). Assume the kernel $\rho$ satisfies hypotheses (H$\rho$1')-(H$\rho$3), and in addition either hypothesis 
 (H$\rho$4) or (H$\rho$5). Then in addition to hypothesis (H$\rho$4) statement (i) implies statement (ii), and in addition to hypothesis (H$\rho$5) statement (i) implies statement (iii) below.
 \begin{itemize}
\item[(i)] There exists a constant $C_4=C_4(n,p,\rho,\BA,\BB)$ such that for every cube $Q\subset\mathbb R^n$ such that $4Q\subseteq\GO,$ the inequality on concentric cubes
\begin{equation}
\label{2.3}
\int_{Q}\int_{Q}\rho(x,y)|\BA(Q,f)(x)-\BA(Q,f)(y)|^pdxdy\leq C_4\int_{4Q}\int_{4Q}\rho(x,y)|\BB f(x,y)|^pdxdy
\end{equation}
holds for all $f\in L^1_{loc}(\GO,\mathbb R^N).$  
\item[(ii)] There exists a constant $\widetilde C_4=\widetilde  C_4(n,p,\rho,\GO,\BA,\BB)$ and a non-degenerate cube 
$Q\subset \GO,$ such that the inequality
\begin{align}
\label{2.4}
\int_{\Omega}\int_{\Omega}\rho(x,y)|\BA(Q,f)(x)&-\BA(Q,f)(y)|^pdxdy
\leq \widetilde  C_4\int_{\Omega}\int_{\Omega}\rho(x,y)|\BB f(x,y)|^pdxdy,
\end{align}
holds for all $f\in L^1_{loc}(\GO,\mathbb R^N).$ 
\item[(iii)] There exists a constant $\widehat C_4=\widehat  C_4(n,p,\rho,\GO,\BA,\BB)$ and a non-degenerate cube 
$Q\subset \GO,$ such that the inequality
\begin{align}
\label{2.5}
\int_{\Omega}\int_{\Omega}\rho(x,y)&|\BA(Q,f)(x)-\BA(Q,f)(y)|^pdxdy\\ \nonumber
&\leq \widehat  C_4\int_{\Omega}\int_{\Omega}\rho(x,y)|\BB f(x,y)|^pdxdy+ 
\widehat  C_4\int_{\Omega}|\BA(Q,f)(x)|^pdx.
\end{align}
holds for all $f\in L^1_{loc}(\GO,\mathbb R^N).$ 
\end{itemize}

\end{theorem}

\vspace{0.1cm}

In the proof of Theorem~2.3 we will need to employ a truncated seminorm equivalence theorem to reduce the full double integral to 
the Carleson boxes. We have the following theorem, which is in the spirit of 
[\ref{bib:Seeger},\ref{bib:Pra.Sak.},\ref{bib:Rutkowski.1}]. The same domain reduction is proven in
 [\ref{bib:Rutkowski.1}, Theorem~1.1] for general radial kernels, and in [\ref{bib:Pra.Sak.}, Theorem~1.6] for $1<p<\infty,$ and in [\ref{bib:Seeger}, Corollary 1] for $1\leq p<\infty$ for the standard kernel. The result with the most general kernel in [\ref{bib:Rutkowski.1}, Theorem~1.1], where the truncated seminorm equivalence for Triebel-Lizorkin seminorms is proven with $1<q\leq p<\infty$ and with general kernels satisfying some monotonicity and summability conditions. We will prove 
 both the truncated norm and seminorm equivalence variants of the theorem for the case $1\leq p=q<\infty$ (including the 
 endpoint case $p=1$) for kernels satisfying hypotheses (H$\rho$1') and (H$\rho$2') and either of the moment conditions 
 (H$\rho$4) or (H$\rho$5). The same proof as in [\ref{bib:Rutkowski.1}] using duality has to go through with adjustments in 
 principle, but we will provide a different direct seminorm estimation proof with a geometric component. 

\begin{theorem}[Norm and seminorm truncation]
\label{Thm:2.4}
Denote the distance function $\Gd(x)=\mathrm{dist}(x,\partial\GO)$ for $x\in\GO$ and assume hypotheses (H$\rho$1'), 
(H$\rho$2'). Then for every $\Gth\in (0,1),$ one has the equivalence of the seminorms and norms:
\begin{itemize}
\item[(i)] Under the additional hypothesis (H$\rho$4), one has 
\begin{align}
\label{2.6}
\int_{\GO}\int_{\GO}\rho(x,y)|f(x)-f(y)|^pdydx \leq C\int_{\GO}\int_{B(x,\Gth\delta(x))}\rho(x,y)|f(x)-f(y)|^pdydx,
\end{align}
for all $f\in L^1_{loc}(\GO).$ The constant $C$ depends only on $n,p,c_1,c_2,c_2',c_4,\Gth,$ and $\GO.$ 
\item[(ii)] Under the additional hypothesis (H$\rho$5), one has 
\begin{align}
\label{2.7}
\int_{\GO}\int_{\GO}\rho(x,y)&|f(x)-f(y)|^pdydx\\ \nonumber
& \leq C\int_{\GO}\int_{B(x,\Gth\delta(x))}\rho(x,y)|f(x)-f(y)|^pdydx+C\int_{\GO}|f(x)|^pdx
\end{align}
for all $f\in L^1_{loc}(\GO).$ The constant $C$ depends only on $n,p,c_1,c_2,c_2',c_5,\Gth,$ and $\GO.$ 
\end{itemize}
\end{theorem}

\vspace{0.1cm}

\begin{remark}
\label{Rem.2.2}
It is easy to verify, that given hypotheses (H$\rho1'$)-(H$\rho2'$), hypothesis (H$\rho3$) may be relaxed to hold only locally for 
$a\in (0,a_0)$ for some $a_0>0,$ and in hypotheses (H$\rho4$) and (H$\rho5$), the integration can be relaxed to be on 
$(a,a_0)$ for $0<a<a_0.$   
\end{remark}


\section{Proof of main results}
\setcounter{equation}{0}
\label{Sec:3}

\subsection{Singularity removal on the cube}

\begin{proof}[Proof of Theorem~2.1]
We can assume without loss of generality that $\BB(E)<\infty,$ where we set (omitting $f$ when no ambiguity is present) for a bounded measurable set $E\subset\mathbb R^n$ and a measurable field $f\colon E\to\mathbb R^N:$ 
\begin{equation}
\label{3.1}
\BB(E)=\int_{E}\int_{E}\rho(x,y)|\BB f(x,y)|^pdxdy,
\end{equation}
and the unweighted averages
\begin{equation}
\label{3.2} 
m(E)=\fint_E\fint_E |\BB f(x,y)|^p dx dy.
\end{equation}
In what follows, the constant $C$ may only depend on $n,p,N,m_1,m_2,c_1,c_2,c_3',C_3,C_1(\BA),$ and $C_2(\BA),$ where the dependence on the constant $c_3,$ which is an important characteristic of the kernel $\rho,$ will be shown explicitly. This is important in applications, because $c_3=\frac{1}{p(1-s)}$ for the standard kernel (\ref{1.18}), for which we will get an explicit asymptotics of the constant in nonlocal inequalities with the standard singular weight. Also, whenever necessary, we will provide additional details of why the corresponding steps are valid for $p=1$ and how the constants are affected (or not affected). Let $\ell_P$ denote the side length of the cube $P,$ where we will abbreviate for notational simplicity $\ell:=\ell_Q$ for the fixed cube $Q$ in consideration. For any positive integer $j,$ we divide $4Q$ into $2^{(j+2)n}$ dyadic cubes of size $2^{-j}\ell$ each, and denote by 
$\mathcal J_j$ the family of those (generation-$j$) sub-cubes of $4Q.$ Let also $\mathcal I_j$ be the generation-$j$ sub-cubes of $4Q$ meeting $Q$. We will use the lowercase letter $q$ to denote a dyadic sub-cube. It is clear that if $j \geq 3$ then $12q\subset  4Q$ for all $q\in\mathcal I_j.$ For every such cube $q$, we have by (\ref{2.1}) the inequality
\begin{equation}
\label{3.3}
\int_{3q}|\BA(3q,f)(x)|^pdx\leq C|q|m(12q).
\end{equation}

 We divide the proof into 3 steps. \\
 
\noindent\textbf{Step 1. Decomposition by distance.}
We separate the pairs with $|x-y|>2^{-3}\ell$ from those at smaller distances. By Minkowski, the inequality (\ref{2.1}), and the hypotheses (H$\rho$1')-(H$\rho$2'), we have the obvious estimate on the large-distance product set 
\begin{align}
\label{3.4}
 \int_{Q} \int_{Q\cap \{|x-y|>2^{-3}\ell\}} & \rho(x-y)|\BA(Q,f)(x)-\BA(Q,f)(y)|^pdydx \\ \nonumber
& \leq  C\rho\left(\frac{\ell}{8}\right)\int_{Q} \int_{Q\cap \{|x-y|>2^{-3}\ell\}}\left(|\BA(Q, f)(x)^p+|\BA(Q, f)(y)|^p\right)dxdy \\ \nonumber
& \leq  \frac{C}{\rho(4\sqrt{n}\ell)}\rho\left(\frac{\ell}{8}\right)\int_{4Q} \int_{4Q}\rho(x-y)|\BB f(x,y)|^pdxdy\\ \nonumber  
 & \leq C\int_{4Q} \int_{4Q}\rho(x-y)|\BB f(x,y)|^pdxdy.
\end{align}
To deal with the remaining pairs, for $J\ge 3$ define the "dyadic" $x-$annuli
\begin{equation}
\label{3.5}
 S_J(x)=\{y\in Q \ :  \ 2^{-J-1}\ell <|x-y|\le2^{-J}\ell  \},
\end{equation}
where we have a definite sharp control of the weight. For almost every $x\in Q$, let $q_J(x)$ be the unique generation-$J$ dyadic 
cube containing $x$ (all dyadic boundaries together have measure zero). If $q=q_J(x)$ and $y\in S_J(x)$, then obviously 
$y\in3q_J.$ On this pair we can decompose 
\begin{align*}
\BA&(Q, f)(x)-\BA(Q, f)(y)\\
=&\big[(\BA(Q, f)-\BA(3q_J,f))(x)-(\BA(Q, f)-\BA(3q_J,f))(y)\big]+\BA(3q_J,f))(x)-\BA(3q_J,f))(y),
\end{align*}
Consequently, we have by Minkowski inequality
\begin{align}
\label{3.6}
|\BA(Q, f)(x)&-\BA(Q, f)(y)|^p\\ \nonumber
&\leq 3^{p-1}|\BGd_{Q,3q_J}(x)-\BGd_{Q,3q_J}(y)|^p+3^{p-1}|\BA(3q_J,f)(x)|^p+3^{p-1}|\BA(3q_J,f)(y)|^p.
\end{align}
We estimate the contribution of each summand above separately. The estimation of the second and the third summands is 
similar and we do that first. \\

\noindent\textbf{Step 2. Bounding the local terms.} On $S_J(x)$ we have by (\ref{3.3}) and the quasi-monotonicity of $\rho,$ that
\begin{align}
\label{3.7}
 \int_{q_J}\int_{S_J(x)} \rho(x-y)|\BA(3q_J,f)(x)|^p dy dx & \leq C|3q_J| \rho(2^{-J}\ell)\int_{3q_J} |\BA(3q_J,f)(x)|^pdx\\ \nonumber
 &\leq C\rho(2^{-J}\ell)|q_J|^2m(12q_J).
 \end{align}
and 
\begin{align}
\label{3.8}
 \int_{q_J}\int_{S_J(x)} \rho(x-y)|\BA(3q_J,f)(y)|^p dy dx & \leq C|q_J| \rho(2^{-J}\ell)\int_{3q_J} |\BA(3q_J,f)(y)|^pdy\\ \nonumber
   &\leq C\rho(2^{-J}\ell)|q_J|^2m(12q_J),
    \end{align}
thus summing (\ref{3.7}) and (\ref{3.8}) over $q\in\mathcal I_J,$ the local-error contribution at distance $2^{-J}\ell$ 
is at most $CM_J,$ where
\begin{equation}
\label{3.9}
 M_J:=\sum_{q_J\in\mathcal I_J} \rho(2^{-J}\ell)|q_J|^2m(12q_J).
 \end{equation}
 Now we will prove that 
 \begin{equation}
\label{3.10}
L_1:= \sum_{j=3}^\infty M_j \leq C \BB(4Q).
 \end{equation}
 We can estimate the sum in (\ref{3.10}) by changing the order of the nonnegative summation and utilizing the bounded overlap 
 of the cubes $12q.$ The counting function below will be useful for that purpose along the calculations, thus we define 
 \begin{equation}
\label{3.11}
N_j(x,y):=\sum_{q\in \mathcal{I}_j} \mathds{1}_{12q}(x) \mathds{1}_{12q}(y),
 \end{equation}
 which is bounded due to the bounded overlap property of the enlarged $12q$ cubes:
  $$
 N_j(x,y)\leq C,\qquad j\in\mathbb N.
 $$
 Assume now $k\geq 3$ is the last generation that satisfies $|x-y|\leq 12\sqrt{n}2^{-k}\ell,$ i.e., we have 
 $$
 12\sqrt{n}2^{-k-1}\ell<|x-y|\leq 12\sqrt{n}2^{-k}\ell,
 $$
 thus for those pairs $(x,y)$ both points $x$ and $y$ can't belong to the same $12q_j$ cube with $j>k$, hence we have on one hand 
 $$
 N_j(x,y)=0\qquad\text{for}\qquad j>k,
 $$
 and hence due to the bound $N_j(x,y)\leq C,$ and the inverse doubling inequality (H$\rho$2'), the estimate 
 \begin{align*}
 \sum_{j=3}^\infty N_j(x,y)\rho(2^{-j}\ell)&=\sum_{j=3}^k N_j(x,y)\rho(2^{-j}\ell)\\ 
 &\leq C\sum_{j=3}^k c_2'^{k-j}\rho(2^{-k}\ell)\\
 & \leq \frac{C}{1-c_2'}\rho(2^{-k}\ell).
 \end{align*}
 And on the other hand we have again by (H$\rho$1')-(H$\rho$2') and thanks to $12\sqrt{n}2^{-k-1}\ell<|x-y|\leq 12\sqrt{n}2^{-k}\ell,$
 the bound
 $$
 \rho(2^{-k}\ell)\leq C\rho(x-y).
 $$ 
 Consequently, we can estimate 
 \begin{align}
\label{3.12}
 \sum_{j=3}^\infty M_j& =12^{-2n}\sum_{j=3}^\infty  \int_{4Q}\int_{4Q} N_j(x,y)\rho(2^{-j}\ell)|\BB f(x,y)|^p dx dy\\ \nonumber
 & = 12^{-2n}\int_{4Q}\int_{4Q} \left\{\sum_{j=3}^\infty N_j(x,y)\rho(2^{-j}\ell)\right\}|\BB f(x,y)|^p dx dy\\ \nonumber
&\le C\int_{4Q}\int_{4Q} \rho(x-y)|\BB f(x,y)|^p dx dy\\  \nonumber
 &=C \BB(4Q),
  \end{align}
 which is (\ref{3.10}).\\
 
 \noindent\textbf{Step 3. Bounding the global-to-local connection terms.}  Now we estimate the global-to-local connection contributions coming from (\ref{3.6}), i.e., the summand $|\BGd_{Q,3q}(x)-\BGd_{Q,3q}(y)|^p,$ for which we need to estimate for $x\in q=q_J(x)$ and $y\in S_J(x),$ the integrals
\begin{align*}
 L_2&=\int_q\int_{S_J(x)} \rho(x-y)|\BGd_{Q,3q_J}(x)-\BGd_{Q,3q_J}(y)|^pdxdy.\\
  \end{align*}
To estimate these differences using local information, we telescope along the tripled dyadic ancestors of $q_J(x)$:
\begin{align*}
 & |\BGd_{Q,3q_J}(x)-\BGd_{Q,3q_J}(y)|\leq  |\BGd_{Q,3q_3}(x)-\BGd_{Q,3q_3}(y)|+|\BGd_{3q_3,3q_J}(x)-\BGd_{3q_3,3q_J}(y)|\\
&\leq |\BGd_{Q, 3q_3}(x)-\BGd_{Q, 3q_3}(y)|+\sum_{j=4}^{J} |\BGd_{3q_j,3q_{j-1}}(x)-\BGd_{3q_j,3q_{j-1}}(y)|,
 \end{align*}
where $q_{j-1}$ is the ancestor (parent) of $q_j,$ that is the previous generation dyadic cube such that $q_j\subset q_{j-1}.$ Consequently, we can estimate by Minkowski and (H$\rho$1')-(H$\rho$2'):
\begin{align*}
 L_2(J,q_J):=\int_{q_J}\int_{S_J(x)} \rho(x-y)| & \BGd_{Q,3q_J}(x)-\BGd_{Q,3q_J}(y)|^pdxdy\\
 & \leq C\rho(2^{-J}\ell) \int_{q_J} \int_{S_J(x)}|\BGd_{Q,3q_3}(x)-\BGd_{Q, 3q_3}(y)|^pdxdy \\
 &+C\rho(2^{-J}\ell)\int_{q_J}\int_{S_J(x)}\left(\sum_{j=4}^{J} |\BGd_{3q_j,3q_{j-1}}(x)-\BGd_{3q_j,q3_{j-1}}(y)|\right)^pdxdy\\
 &:=C(I_1+I_2).
  \end{align*}
  For $I_1$ we have by hypothesis (HA1), that
 \begin{align*} 
 I_1(J,q_J)&=\rho(2^{-J}\ell) \int_{q_J} \int_{S_J(x)}|\BGd_{Q,3q_3}(x)-\BGd_{Q, 3q_3}(y)|^pdxdy \\
 & \leq \rho(2^{-J}\ell) \int_{q_J} \int_{3q_J}|\BGd_{Q,3q_3}(x)-\BGd_{Q, 3q_3}(y)|^pdxdy \\
    &\leq C 2^{-Jp}\rho(2^{-J}\ell) \frac{|q_J|^2}{|Q|^2}  \int_{Q} \int_{3q_3}|\BGd_{Q,3q_3}(x)-\BGd_{Q, 3q_3}(y)|^pdxdy \\
    &=C2^{-J(p+2n)}\rho(2^{-J}\ell) \int_{Q} \int_{3q_3}|\BGd_{Q,3q_3}(x)-\BGd_{Q, 3q_3}(y)|^pdxdy.
         \end{align*}  
Thus for a fixed $J\geq 3,$ the summation in $q_J$ gives 
 \begin{align*} 
 \sum_{q_J}I_1(J,q_J) 
 & \leq \sum_{q_J} C 2^{-J(p+2n)}\rho(2^{-J}\ell)\int_{Q} \int_{3q_3}|\BGd_{Q,3q_3}(x)-\BGd_{Q, 3q_3}(y)|^pdxdy\\
 &= C 2^{-J(p+n)}\rho(2^{-J}\ell)\int_{Q} \int_{3q_3}|\BGd_{Q,3q_3}(x)-\BGd_{Q, 3q_3}(y)|^pdxdy.
 \end{align*}  
 Now we need to perform another summation in $J.$ To estimate that sum, we divide the interval $(0,\ell]$ into the dyadic intervals 
 $(2^{-J-1}\ell, 2^{-J}\ell],$ and apply the estimate $\rho(t) \geq c_1c_2 \rho(2^{-J-1}\ell)$ for 
 $t\in (2^{-J-1}\ell, 2^{-J}\ell],$ to arrive at 
  \begin{align*} 
\sum_{J=3}^\infty \sum_{q_J}I_1(J,q_J) 
&\leq \sum_{J=3}^\infty C2^{-J(p+n)}\rho(2^{-J}\ell)\int_{Q} \int_{3q_3}|\BGd_{Q,3q_3}(x)-\BGd_{Q, 3q_3}(y)|^pdxdy\\
&\leq C\cdot \left(\frac{1}{\ell^{n+p}\rho(\ell)}\int_0^{\ell} t^{n+p-1}\rho(t)dt \right)
\int_{Q}\int_{3q_3}\rho(\ell)|\BGd_{Q,3q_3}(x)-\BGd_{Q, 3q_3}(y)|^pdxdy\\
&\leq Cc_3\int_{Q}\int_{3q_3}\rho(\ell)|\BGd_{Q,3q_3}(x)-\BGd_{Q, 3q_3}(y)|^pdxdy,
\end{align*}  
where we invoked (H$\rho$3) to replace the integral of $\rho$ in the last step. Further choosing $Q,$ $P=3q_3,$ and 
$P_1=P2=q_3$ in (HA2), where all four cubes have comparable size, we have the replacement
 \begin{align*} 
 \int_{Q} \int_{3q_3}\rho(\ell)|\BGd_{Q,3q_3}(x)-\BGd_{Q, 3q_3}(y)|^pdxdy 
 \leq C\int_{q_3} \int_{q_3}\rho(\ell)|\BGd_{Q,3q_3}(x)-\BGd_{Q, 3q_3}(y)|^pdxdy, 
  \end{align*} 
  where 
$$
  \BGd_{Q,3q_3}(x)-\BGd_{Q, 3q_3}(y)=\BA(Q,f)(x)-\BA(Q,f)(y)-\BA(3q_3,f)(x)+\BA(3q_3,f)(y),
$$
which can be estimated by Minkowski and then estimating all four integrals over $q_3\times q_3$ separately exactly in 
the same way utilizing (\ref{2.1}). For instance we can calculate
 \begin{align*} 
  \int_{q_3} \int_{q_3}\rho(\ell)|\BA(Q,f)(x)|^pdxdy&=|q_3|\int_{q_3}\rho(\ell)|\BA(Q,f)(x)|^pdx\\
  &\leq |q_3|\int_{Q}\rho(\ell)|\BA(Q,f)(x)|^pdx\\
    &\leq C\int_{4Q}\int_{4Q}\rho(x-y)|Bf(x,y)|^pdxdy.
 \end{align*}  
 due to the inverse doubling and quasi-monotonicity properties of $\rho.$ Consequently, It remains to estimate the sum corresponding to 
\begin{equation}
\label{3.13}
I_2(J,q_J)=\rho(2^{-J}\ell)\int_{q_J}\int_{S_J(x)}\left(\sum_{j=4}^{J} |\BGd_{3q_j,3q_{j-1}}(x)-\BGd_{3q_j,3q_{j-1}}(y)|\right)^pdxdy.
\end{equation}
which is the more complex part. We will apply Minkowski inequality first to get rid of the $p-$th power of the sum under the 
integral. We have 
\begin{align}
\label{3.14}
I_2(J,q_J)&=\rho(2^{-J}\ell)\int_{q_J}\int_{S_J(x)}\left(\sum_{j=4}^{J} |\BGd_{3q_j,3q_{j-1}}(x)-\BGd_{3q_j,3q_{j-1}}(y)|\right)^pdxdy\\ 
\nonumber
&\leq \rho(2^{-J}\ell)\left(\sum_{j=4}^{J}
\left( \int_{q_J}\int_{S_J(x)}|\BGd_{3q_j,3q_{j-1}}(x)-\BGd_{3q_j,3q_{j-1}}(y)|^pdydx\right)^{\frac{1}{p}}\right)^p\\ \nonumber
& \leq \rho(2^{-J}\ell)\left(\sum_{j=4}^{J}
\left( \int_{q_J}\int_{3q_J}|\BGd_{3q_j,3q_{j-1}}(x)-\BGd_{3q_j,3q_{j-1}}(y)|^pdydx\right)^{\frac{1}{p}}\right)^p.
\end{align}
Next we  apply (HA1) and (\ref{2.1}) to estimate the contribution 
\begin{align*}
\int_{q_J}\int_{3q_J}&|\BGd_{3q_j,3q_{j-1}}(x)-\BGd_{3q_j,3q_{j-1}}(y)|^pdydx\\ 
&\leq C2^{(p+2n)(j-J)}\int_{3q_j}\int_{3q_j}|\BGd_{3q_j,3q_{j-1}}(x)-\BGd_{3q_j,3q_{j-1}}(y)|^pdydx\\
&\leq C2^{(p+2n)(j-J)}|q_j|\int_{3q_j}|(\BA(3q_j,f)(x)|^p+|\BA(3q_{j-1},f)(x)|^p)dx\\
&\leq C2^{(p+2n)(j-J)}|q_j|\left(\int_{3q_j}|(\BA(3q_j,f)(x)|^pdx+\int_{3q_{j-1}}|\BA(3q_{j-1},f)(x)|^pdx\right)\\
&\leq C2^{(p+2n)(j-J)}|q_{j-1}|^2m(12q_{j-1})\\
&=C2^{p(j-J)}|q_{J-1}|^2m(12q_{j-1})\\
\end{align*}
Consequently, we get 
\begin{align}
\label{3.15}
I_2(J,q_J)& \leq C\rho(2^{-J}\ell)|q_{J-1}|^2\left(\sum_{j=4}^{J} 
2^{j-J}m(12q_{j-1})^{\frac{1}{p}}\right)^p\\ \nonumber
&\leq C\rho(2^{-J}\ell)|q_J|^2\left(\sum_{j=3}^{J}
2^{j-J}m(12q_j)^{\frac{1}{p}}\right)^p
\end{align}

To simplify further, next we apply H\"older's inequality to the sum $\sum_{j=3}^{J}2^{j-J}m(12q_j)^{\frac{1}{p}}$ with the 
exponents $p$ and $q=\frac{p}{p-1}$ for the case $p>1,$ and for the case $p=1$ this H\"older's inequality application step is not needed. A crucial point is that we need to maintain enough negative powers of $2$ to ensure enough decay and convergence 
after the inequality application. Thus we fix a small constant $\Ge>0$ yet to be chosen, and estimate for $p>1$ as follows:  
\begin{align*}
\sum_{j=3}^{J} 2^{j-J}m(12q_j)^{\frac{1}{p}}
&=\sum_{j=3}^{J} \left(2^{\frac{\Ge(j-J)}{q}}\right)\cdot \left(2^{\frac{\Ge(j-J)}{p}+(1-\Ge)(j-J)}m(12q_j)^{\frac{1}{p}}\right)\\ \nonumber
&\leq \left(\sum_{j=3}^{J} 2^{\Ge(j-J)}\right)^{\frac{1}{q}} \left(\sum_{j=3}^{J} 2^{(j-J)(p+\Ge(1-p))} m(12q_j)\right)^{\frac{1}{p}}\\ \nonumber
&\leq \frac{2^{\frac{\Ge(p-1)}{p}}}{(2^{\Ge}-1)^{\frac{p-1}{p}}} \left(\sum_{j=3}^{J} 2^{(j-J)(p+\Ge(1-p))} m(12q_j)\right)^{\frac{1}{p}}\\
\end{align*}
and clearly the estimate holds true for the case $p=1$ too. Therefore we can simplify (\ref{3.15}) further as 
\begin{equation}
\label{3.16}
I_2(J,q_J) 
\leq CC(\Ge,p)\rho(2^{-J}\ell)|q_J|^2\sum_{j=3}^{J}2^{(j-J)(p+\Ge(1-p))} m(12q_j),
\end{equation}
where for keeping track of the constant we record 
\begin{equation}
\label{3.17}
C(\Ge,p)=\frac{2^{\Ge(p-1)}}{(2^\Ge-1)^{p-1}}.
\end{equation}
We can further wright 
\begin{align*}
\rho(2^{-J}\ell)|q_J|^2&\sum_{j=3}^{J}2^{(j-J)(p+\Ge(1-p))} m(12q_j)\\
&=\rho(2^{-J}\ell)|q_J|^2\sum_{j=3}^{J}2^{(j-J)(p+\Ge(1-p))} \frac{1}{|12q_j|^2}\int_{12q_j}\int_{12q_j}|\BB f(x,y)|^pdxdy\\
&=12^{-2n}\rho(2^{-J}\ell)\sum_{j=3}^{J}2^{(2n+p+\Ge(1-p))(j-J)}\int_{12q_j}\int_{12q_j}|\BB f(x,y)|^pdxdy.\\
\end{align*}
For a fixed $J\geq 3,$ summing over $q_J$ gives from (\ref{3.16}) and from the above estimate 
\begin{align*}
\sum_{q_J}I_2(J,q_J)
& \leq C C(\Ge,p)\int_{4Q}\int_{4Q}\left\{\sum_{j=3}^J 2^{(n+p+\Ge(1-p))(j-J)}N_j(x,y)\rho(2^{-J}\ell)\right\}|\BB f(x,y)|^pdxdy.
\end{align*}
This is because each time climbing a step, a parent cube must be counted exactly $2^n$ times, thus in the $k-th$ step, the parent ancestor cube has to be counted exactly $2^{nk}$ times. Now, another summation in $J$ yields
\begin{align}
\label{3.18}
\sum_{J=3}^\infty & \sum_{q_J}I_2(J,q_J)\\ \nonumber
& \leq C C(\Ge,p)\sum_{J=3}^\infty \int_{4Q}\int_{4Q}\left\{\sum_{j=3}^J 2^{(n+p+\Ge(1-p))(j-J)}N_j(x,y)\rho(2^{-J}\ell)\right\}
|\BB f(x,y)|^pdxdy\\  \nonumber
&=C C(\Ge,p)\int_{4Q}\int_{4Q}\left\{\sum_{J=3}^\infty\sum_{j=3}^J 2^{(n+p+\Ge(1-p))(j-J)}N_j(x,y)\rho(2^{-J}\ell)\right\}
|\BB f(x,y)|^pdxdy.
\end{align}
Consequently, if we assume  as above for a pair $(x,y)\in 4Q\times 4Q$ the bounds 
$$
 12\sqrt{n}2^{-k-1}\ell<|x-y|\leq 12\sqrt{n}2^{-k}\ell,
 $$
 hold for some $k\geq 3,$ then for those pairs $(x,y)$ we have 
 $$
 N_j(x,y)=0\qquad\text{for}\qquad j>k,
 $$
thus utilizing the boundedness of $N_j$ and the inverse doubling property and quasi-monotonicity (H$\rho$1')-(H$\rho$2') 
of $\rho,$ we can estimate setting for convenience $b=2^{n+p+\Ge(1-p)},$ as follows:
\begin{align}
\label{3.19}
\sum_{J=3}^\infty\sum_{j=3}^J b^{j-J}N_j(x,y)\rho(2^{-J}\ell)\leq \sum_{m=0}^\infty C_m \rho(2^{-m}\ell),
 \end{align}
 where 
\begin{equation}
\label{3.20}
C_m\leq 
\begin{cases}
C \sum_{j=0}^{m} b^{j-m}\leq C &  \text{if} \ \ m\leq k \\[2ex]
C \sum_{j=0}^{k} b^{j-m}\leq Cb^{k-m} &  \text{if} \ \ m> k.
\end{cases}
\end{equation}
Consequently, we get from (\ref{3.19}) and (\ref{3.20}) the bound 
\begin{equation}
\label{3.21}
\sum_{J=3}^\infty\sum_{j=3}^J b^{j-J}N_j(x,y)\rho(2^{-J}\ell) \leq C \sum_{m=0}^k \rho(2^{-m}\ell)+
 C \sum_{i=1}^\infty b^{-i}\rho(2^{-k-i}\ell).
 \end{equation}
The first sum can be estimated easily by the inverse doubling property (H$\rho$2') as follows:
\begin{align*}
\sum_{m=0}^k \rho(2^{-m}\ell) & \leq \sum_{m=0}^k c_2'^{k-m}\rho(2^{-k}\ell)\\
& \leq\frac{1}{1-c_2'}\rho(2^{-k}\ell)\\
& \leq\frac{C}{1-c_2'}\rho(x-y),
\end{align*}
while for the second sum we need a more sophisticated argument. Denote 
$$
\varphi(t)=\int_0^t s^{n+p-1}\rho(s)ds,\qquad t>0,
$$
in the moment bound (H$\rho$3). We can calculate for fixed $0<\mu<a$ (where the integrals are convergent) and
 for $\delta>0$ by integrating by parts and owing back to (H$\rho$3), that
\begin{align*}
\int_\mu^a t^{n+p-1-\delta}\rho(t)dt &=\int_\mu^a t^{-\delta}\varphi'(t)dt\\
&\leq a^{-\delta}\varphi(a)+\delta \int_\mu^a t^{-\delta-1}\varphi(t)dt\\
&\leq a^{-\delta}\varphi(a)+\delta c_3\int_\mu^a t^{n+p-\delta-1}\rho(t)dt,\\
\end{align*}
which implies together with $c_3\delta=\frac{1}{2}$ for the specific choice $\delta=\frac{1}{2c_3},$ the estimate
\begin{align*}
\int_\mu^a t^{n+p-1-\delta}\rho(t)dt & \leq 2a^{-\delta}\varphi(a)\\
&\leq 2c_3 a^{n+p-\delta}\rho(a),
\end{align*}
for all $0<\mu\leq 1.$ Therefore sending $\mu$ to zero, we discover the key improved moment bound with $\Gd-$margin:
\begin{align}
\label{3.22}
\int_0^a t^{n+p-1-\delta}\rho(t)dt \leq 2c_3 a^{n+p-\delta}\rho(a).
 \end{align}
 \begin{remark}
 A bootstrapping argument yields the bound 
$$
\int_0^a t^{n+p-1-\delta_k}\rho(t)dt \leq 2^kc_3 a^{n+p-\delta_k}\rho(a),\qquad \text{with} \qquad \delta_k=\frac{1}{c_3}(1-2^{-k}),\ \ k=1,2,\dots.
$$
 However, the estimate for $k=1$ in (\ref{3.22}) is sufficient for our purpose of bounding the second sum in (\ref{3.21}).  
 \end{remark}
We now choose $a=2^{-k}\ell$ in (\ref{3.22}) and divide the interval $(0,2^{-k}\ell]$ into dyadic intervals 
$(2^{-k-i}\ell,2^{-k-i+1}\ell],$ for $i=1,2\dots,$ and utilize (H$\rho$1') and (H$\rho$2') to estimate 
\begin{align*}
\sum_{i=1}^\infty 2^{-(i+k)(n+p-\delta)}\ell^{n+p-\Gd}\rho(2^{-k-i}\ell) & \leq C\int_0^{2^{-k}\ell} t^{n+p-1-\Gd}\rho(t)dt\\ 
& \leq Cc_3 2^{-k(n+p-\Gd)} \ell^{n+p-\Gd}\rho(2^{-k}\ell),
\end{align*}
which yields anther key bound 
\begin{equation}
\label{3.23}
\sum_{i=0}^\infty 2^{-i(n+p-\Gd)}\rho(2^{-k-i}\ell) \leq Cc_3\rho(2^{-k}\ell).
\end{equation}
Now, the choice of $\Ge$ is trivial, namely one should choose
 \begin{equation}
\label{3.24}
 \Ge=\frac{\Gd}{p-1}=\frac{1}{2c_3(p-1)},\qquad p>1.
 \end{equation}
 Note, that for the case $p=1,$ the step with H\"older inequality as well as the steps leading to (\ref{3.22}) are not necessary, and one can proceed to estimating the two modified sums in (\ref{3.21}), and no additional $\Ge(p-1)$ error occurs. Now, we can estimate the second sum thanks to (\ref{3.23}) as follows:
\begin{align*}
 \sum_{i=1}^\infty b^{-i}\rho(2^{-k-i}\ell)&=\sum_{i=1}^\infty 2^{-i(n+p-\Gd)}\rho(2^{-k-i}\ell)\\
 & \leq Cc_3\rho(2^{-k}\ell)\\ 
   & \leq Cc_3 \rho(x-y).
\end{align*}
Finally, inserting this bound into (\ref{3.18}), we discover
\begin{equation}
\label{3.25}
\sum_{J=3}^\infty \sum_{q_J}I_2(J,q_J)\leq  C\int_{4Q}\int_{4Q} \rho(x-y)|\BB f(x,y)|^pdxdy.
\end{equation} 
This completes Step 3. Now, combining Steps 1-3, we arrive at (\ref{2.2}). Note, that the constant we got in (\ref{2.2}) has explicit $c_3$-dependence as 
\begin{align*}
C=Cc_3C(p,\Ge)=Cc_3\frac{2^{\Ge(p-1)}}{(2^\Ge-1)^{p-1}},
\end{align*}
and it is easy to see that it is bounded above by $C(1+c_3^p).$ This completes the proof of Theorem~2.1.

\end{proof}


\subsection{From cubes to uniform domains}
\label{Sub:3.2}

\begin{proof}[Proof of Theorem~2.3]

The strategy of the proof is similar to that of [\ref{bib:Jones},\ref{bib:Pra.Tos.},\ref{bib:Pra.Sak.},\ref{bib:Aco.Dre.Dur.Lop.Oje.}] via a Whitney cube decomposition. In [\ref{bib:Aco.Dre.Dur.Lop.Oje.}], the authors extend the fractional Korn's first inequality from smoothened cubes to uniform domains through Whitney decomposition through a gluing technique, due to the an ordering result on Whitney cubes on John domains, which establishes the arrangement of the cubes into a tree with controlled tails [\ref{bib:Lopez}]. 
We will keep track of the case $p=1$ through the proof lines, as some of the results we use are formulated for $p>1,$ which 
we need to address in the case $p=1$ and verify their validity whenever necessary. We now consider a standard Whitney decomposition of $\GO$ into a countable family of closed dyadic cubes $\mathcal{W}=\{Q_j\}_{j\in \mathcal{J}}$ satisfying 
the conditions (C1)-(C3) in Section~4 (see Section~4 for details). The following lemma on graphs on Whitney cubes, proven by L\'opez-Gar\'cıa in [\ref{bib:Lopez}, Lemma~5.2] is an important ingredient in the proof.\\

\begin{lemma}[{[\ref{bib:Lopez}, Lemma~5.2]}] 
\label{Lem:3.2}
A bounded domain $\GO\subset \mathbb R^n$ is a John domain if and only if given a Whitney decomposition 
$\mathcal{F}=\{Q_j\}_{j\in \mathcal{J}}$ of $\GO,$ there exist a tree structure for the set of indices $ \mathcal{J}$ and a constant 
$K>1$ that satisfy $Q\subseteq KQ_a$ for any pair $(Q_a,Q)$ of cubes, where $Q_a$ is an ancestor of $Q$ through the three structure. In other words, the descendant subtree of $Q_a$ is contained in the expanded set $KQ_a.$ Moreover, the intersection 
of any adjacent parent-child cubes $Q_p$ and $Q_c$ in the three is always an $n-1$ dimensional face of one of these cubes.
\end{lemma}

If this three structure is not possible to maintain after further division of the original cubes as above (thus may be required to increase the constant $C_{\mathcal W}$), we may assume that the three contains repeating vertexes, that can repeat a bounded number of times with the same parent-child and descendant subtree conditions still holding while increasing $K$ if necessary. Let $Q_0$ denote the tree root in the above lemma. Then, we will prove Theorem~2.2 with $Q=2Q_0.$ Initially we fix 
$\Gth=\frac{1}{8\sqrt{n}C_{\mathcal W}+1,}$ so that for any Whitney cube $Q\in \mathcal{F}$ and $x\in Q,$ we have 
by (C1) that 
$$
\Gth\delta(x)\leq 4\Gth C_{\mathcal W}\sqrt{n}\ell_Q<\frac{\ell_Q}{2},
$$
hence we get the Carleson box inclusion $B(x,\Gth\delta(x))\subseteq 2Q.$ Consequently, we can calculate

\begin{align}
\label{3.27}
\int_{\GO}\int_{B(x,\Gth\delta(x))}\rho(x,y)&|\BA(2Q_0,f)(x)-\BA(2Q_0,f)(y)|^pdydx\\  \nonumber
 &=\sum_{j=0}^\infty \int_{Q_j}\int_{B(x,\Gth\delta(x))}\rho(x,y)|\BA(2Q_0,f)(x)-\BA(2Q_0,f)(y)|^pdydx\\   \nonumber
 &\leq \sum_{j=0}^\infty \int_{2Q_j}\int_{2Q_j}\rho(x,y)|\BA(2Q_0,f)(x)-\BA(2Q_0,f)(y)|^pdydx.
\end{align}
Let for $j\in\mathbb N,$ the chain of cubes $Q_j=P_k\to P_{k-1}\to\dots\to P_0=Q_0$ be the ancestor chain leading back to 
the root cube $Q_0$ in the tree in Lemma~3.2. Then for $(x,y)\in (2Q_j)^2=(2P_k)^2,$ we write 
 
\begin{align*}
|\BA(2Q_0,f)&(x)-\BA(2Q_0,f)(y)| \\
&\leq |\BA(2P_k,f)(x)-\BA(2P_k,f)(y)|+| \BGd(2P_k,2P_0)(x)-\BGd(2P_k,2P_0)(y)|\\ 
&\leq |\BA(2P_k,f)(x)-\BA(2P_k,f)(y)|+\sum_{m=0}^{k-1}| \BGd(2P_{m+1},2P_m)(x)-\BGd(2P_{m+1},2P_m)(y)|,
\end{align*}
thus for any fixed $j\in\mathbb N,$ we get by Minkowski
\begin{align}
\label{3.28}
\int_{2Q_j}\int_{2Q_j}&\rho(x,y)|\BA(2Q_0,f)(x)-\BA(2Q_0,f)(y)|^pdydx\\ \nonumber
&\leq  2^{p-1}\int_{2Q_j}\int_{2Q_j}\rho(x,y)|\BA(2Q_j,f)(x)-\BA(2Q_j,f)(y)|^pdydx\\  \nonumber
&+2^{p-1}\int_{2Q_j}\int_{2Q_j}\rho(x,y)\left(\sum_{m=0}^{k-1}| \BGd(2P_{m+1},2P_m)(x)-\BGd(2P_{m+1},2P_m)(y)|\right)^pdydx
\end{align}
For the first summand we have by (\ref{2.3}) the bound
\begin{align}
\label{3.29}
\int_{2Q_j}\int_{2Q_j}\rho(x,y)|\BA(2Q_j,f)(x)-\BA(2Q_j,f)(y)|^pdydx\leq C \int_{8Q_j}\int_{8Q_j}\rho(x,y)|\BB f(x,y)|^pdydx,
\end{align}
where the inclusion $8Q_j\subset \GO$ follows from (C1) (we make initially $C_{\mathcal W}>8$). We can estimate the second summand by Minkowski as follows:
\begin{align}
\label{3.30}
\int_{2Q_j}\int_{2Q_j}&\rho(x,y)\left(\sum_{m=0}^{k-1}| \BGd(2P_{m+1},2P_m)(x)-\BGd(2P_{m+1},2P_m)(y)|\right)^pdydx \\ \nonumber
&\leq \left(\sum_{m=0}^{k-1} \left(\l\int_{2Q_j}\int_{2Q_j} \rho(x,y)|\BGd(2P_{m+1},2P_m)(x)-\BGd(2P_{m+1},2P_m)(y)|^pdydx\right)
^{\frac{1}{p}}\right)^p
 \end{align}
Now, in order to estimate the right-hand side of (\ref{3.30}), we need to transfer the integral from $2Q_j\times 2Q_j$ over 
 $2P_{m+1}\times2P_m$ like in the proof of Theorem~2.1. However, we can't directly utilize (HA1) due to the weight 
 $\rho,$ hence an analogous weighted average Lipschitz estimates are needed, which we derive next. To that end, 
 assume for a moment $P,Q,S\subset\GO$ are cubes and denote for the sake of notational simplicity 
 $g(x)=\BGd(P,Q)(x)$ and $g_0=\fint_S g(z)dz.$ Then we can write for every $x\in S:$
 $$g(x)-g_0=\fint_S (g(x)-g(y))dy,$$
 thus we have by H\"older's inequality
  $$|g(x)-g_0|^p\leq \fint_S |g(x)-g(y)|^pdy,$$
  and integrating one more time in $x$ we get 
 \begin{equation}
\label{3.31}
\fint_S |g(x)-g_0|^p\leq \fint_S\fint_S |g(x)-g(y)|^pdxdy.
 \end{equation}
If we now combine (HA1) together with (\ref{3.31}) with $P_1=P_2=S,$ we discover 
 \begin{equation}
\label{3.32}
 \frac{1}{(\mathrm{diam} S)^p}\fint_S |g(x)-g_0|^p
 \leq \frac{C_1(\BA)}{(\mathrm{diam} (P\cup Q))^p} \fint_P\fint_Q |g(x)-g(y)|^pdxdy.
 \end{equation}
 Since (\ref{3.32}) is true for every cube $S\subset\GO,$ the Campanato Lipschitz characterization yields 
 \begin{equation}
\label{3.33}
\left(\mathrm{Lip}_S (g)\right)^p\leq \frac{C(n,p)C_1(\BA)}{(\mathrm{diam} (P\cup Q))^p} \fint_P\fint_Q |g(x)-g(y)|^pdxdy.
 \end{equation} 
We can also obtain the reverse inequality from (HA2) with a constant depending only on $C_2(\BA)$ and $n.$ 
Consequently, using the average bi-Lipschitz property of $g,$ it is straightforward to obtain the weighted version of the 
average bi-Lipschitz equivalence 
\begin{equation}
\label{3.34}
 \fint_{P_1}\fint_{P_2}\rho(x,y) |g(x)-g(y)|^pdxdy \simeq \frac{\phi(R_1)}{\phi(R)} \fint_P\fint_Q \rho(x,y)  |g(x)-g(y)|^pdxdy,
\end{equation} 
 for $P\cap Q\neq \emptyset$ and $P_1\cap P_2\neq \emptyset,$ where $R_1=\mathrm{diam} P_1\cup P_2,$ $R=\mathrm{diam} P\cup Q,$ and the function $\phi$ is defined as follows: 
 \begin{equation}
\label{3.35}
\phi(r)=r^{-n}\int_0^r t^{n+p-1}\rho(t)dt.
\end{equation} 
Note, that (H$\rho3$) gives on one hand $\phi(r)\leq c_3r^p\rho(r)$ and on the other hand, utilizing 
(H$\rho1$) and (H$\rho2$), we can estimate 
 \begin{align*}
\phi(r)&=r^{-n} \int_0^r t^{n+p-1}\rho(t)dt\\ 
&\geq \frac{1}{c_1}r^{-n} \left(\frac{r}{2}\right)^{n+p}\rho(r)\\
&=Cr^p\rho(r).
\end{align*} 
This gives the simplified weighted equivalence 
\begin{align}
\label{3.36}
 \fint_{P_1}\fint_{P_2} \rho(x,y)|&\BGd(P,Q)(x)-\BGd(P,Q)(y)|^pdxdy \\ \nonumber 
  & \simeq \  \frac{R_1^p\rho(R_1)}{R^p\rho(R)} \fint_P\fint_Q  \rho(x,y)|\BGd(P,Q)(x)-\BGd(P,Q)(y)|^pdxdy,
\end{align} 
with a constant that is between $Cc_3^{-1}$ and $Cc_3.$ Next we choose a cube $S_m\subseteq 2P_{m+1}\cap 2P_m,$ that has a side length comparable to that of $P_m$ and estimate utilizing (\ref{3.36}), (C1)-(C2), and (H$\rho1$)-(H$\rho2$):
\begin{align*}
&\sum_{m=0}^{k-1}\left(\l\int_{2Q_j}\int_{2Q_j} \rho(x,y)|\BGd(2P_{m+1},2P_m)(x)-\BGd(2P_{m+1},2P_m)(y)|^pdydx\right)
^{\frac{1}{p}}\\ 
&\simeq \
\sum_{m=0}^{k-1} 
\left(\frac{\ell_k^{2n+p}\rho(\ell_k)}{\ell_m^{2n+p}\rho(\ell_m)}
\l\int_{S_m}\int_{S_m} \rho(x,y)|\BGd(2P_{m+1},2P_m)(x)-\BGd(2P_{m+1},2P_m)(y)|^p\right)
^{\frac{1}{p}}.
\end{align*} 
Next we can estimate by Minkowski and by (\ref{2.4}) for each $m$ as 
\begin{align*}
 \int_{S_m} & \int_{S_m} \rho(x,y)|\BGd(2P_{m+1},2P_m)(x)-\BGd(2P_{m+1},2P_m)(y)|^pdydx\ \\
&\leq 2^{p-1}  \int_{S_m}\int_{S_m} \rho(x,y)|\BA(2P_{m+1},f)(x)-\BA(2P_{m+1},f)(y)|^pdydx\\\ 
&+2^{p-1}  \int_{S_m}\int_{S_m} \rho(x,y)|\BA(2P_m,f)(x)-\BA(2P_{m},f)(y)|^pdydx\ \\
&\leq 2^{p-1}  \int_{2P_{m+1}}\int_{2P_{m+1}} \rho(x,y)|\BA(2P_{m+1},f)(x)-\BA(2P_{m+1},f)(y)|^pdydx\\\ 
&+2^{p-1}  \int_{2P_m}\int_{2P_m} \rho(x,y)|\BA(2P_m,f)(x)-\BA(2P_{m},f)(y)|^pdydx\ \\
&\leq C(\BB(8P_m)+\BB(8P_{m+1})).
 \end{align*} 
Consequently, we get from (\ref{3.30}) the estimate 
\begin{align}
\label{3.37}
I_j&=\int_{2Q_j}\int_{2Q_j}\rho(x,y)\left(\sum_{m=0}^{k-1} |\BGd(2P_{m+1},2P_m)(x)-\BGd(2P_{m+1},2P_m)(y)|\right)^pdydx \\ \nonumber
&\leq C \left(\sum_{m=0}^{k-1} \left(\frac{\ell_k^{2n+p}\rho(\ell_k)}{\ell_m^{2n+p}\rho(\ell_m)}
[\BB(8P_m)+\BB(8P_{m+1})]\right)^{\frac{1}{p}}\right)^p\\ \nonumber
&=C\ell_k^{2n+p}\rho(\ell_k)\left(\sum_{m=0}^{k-1} \left(\frac{\BB(8P_m)+\BB(8P_{m+1})}{\ell_m^{2n+p}\rho(\ell_m)}
\right)^{\frac{1}{p}}\right)^p\\ \nonumber
\end{align} 
Now we will apply H\"older's inequality in (\ref{3.37}) with the correct weights and with the exponents $q=\frac{p}{p-1}$ and $p,$ 
where in the case $p=1$ no application of the inequality is needed. To that end, let $\Gl_m(j)>0$ for $m=0,\dots,k-1$
yet to be chosen, so that $\sum_{m=0}^{k-1} \Gl_m(j)\leq C$ uniformly in $j.$ Then we can write 
\begin{align*}
\sum_{m=0}^{k-1} \left(\frac{\BB(8P_m)+\BB(8P_{m+1})}{\ell_m^{2n+p}\rho(\ell_m)}\right)^{\frac{1}{p}}
&=\sum_{m=0}^{k-1}\Gl_m(j)^{\frac{1}{q}}\cdot  \left(\frac{\BB(8P_m)+\BB(8P_{m+1})}{\Gl_m(j)^{\frac{p}{q}}\ell_m^{2n+p}\rho(\ell_m)}\right)^{\frac{1}{p}}\\
&\leq \left(\sum_{m=0}^{k-1} \Gl_m(j)\right)^{^{\frac{1}{q}}}\cdot 
\left(\sum_{m=0}^{k-1} \frac{\BB(8P_m)+\BB(8P_{m+1})}{\Gl_m(j)^{\frac{p}{q}}\ell_m^{2n+p}\rho(\ell_m)}\right)^{\frac{1}{p}},
\end{align*}
hence we obtain from (\ref{3.37}) the bound 
\begin{equation}
\label{3.38}
I_j\leq C\ell_k^{2n+p}\rho(\ell_k)\sum_{m=0}^{k-1} \frac{\BB(8P_m)+\BB(8P_{m+1})}{\Gl_m(j)^{p-1}\ell_m^{2n+p}\rho(\ell_m)}.\end{equation} 
Now we fix $\Ge>0$ like we did in the steps between (\ref{3.15}) and (\ref{3.16}) and for a fixed $j$ choose 
$\Gl_m(j)=\frac{\ell_k^\Ge}{\ell_m^\Ge},$ where we need to verify the bound 
\begin{equation}
\label{3.39}
\sum_{m=0}^{k}\frac{\ell_k^\Ge}{\ell_m^\Ge}   \leq C.
\end{equation}
We have by the Whitney cube structure in Lemma~3.2, the descendant subtree volume bound
 \begin{equation}
\label{3.40}
L_m:=\sum_{i=m}^{k} \ell_i^n\leq K^n \ell_m^n,\qquad m=0,1,\dots,k.
\end{equation}
Thus we can calculate
\begin{align*}
L_m & \leq K^n \ell_m^n\\
&= K^n (L_m-L_{m+1}),
\end{align*}
which yields the geometric decay
 \begin{equation}
\label{3.41}
\frac{L_k}{L_m} \leq  \eta^{k-m} ,\qquad m=0,1,\dots,k,\qquad\text{with}\qquad \eta =\frac{K^n-1}{K^n}\in (0,1). 
\end{equation}
Now, we can estimate the sum in (\ref{3.39}) owing to (\ref{3.40}) and (\ref{3.41}) as follows
\begin{align*}
\sum_{m=0}^{k}\frac{\ell_k^\Ge}{\ell_m^\Ge} &\leq  K^{\Ge} \ell_k^\Ge \sum_{m=0}^{k}\frac{1}{L_m^\frac{\Ge}{n}}\\
&= K^{\Ge} \sum_{m=0}^{k}\frac{L_k^\frac{\Ge}{n}}{L_m^\frac{\Ge}{n}}\\
&\leq K^{\Ge} \frac{1}{1-\eta^{\frac{\Ge}{n}}},
\end{align*}
which confirms (\ref{3.39}). It remains to bound the summation over $j$ of $I_j,$ i.e., the right-hand side of
\begin{equation}
\label{3.42}
\sum_{j=1}^\infty I_j\leq C\sum_{j=1}^\infty \ell_k^{2n+p+\Ge(1-p)}\rho(\ell_k)\sum_{m=0}^{k-1} \frac{\BB(8P_m)+\BB(8P_{m+1})}{\ell_m^{2n+p+\Ge(1-p)}\rho(\ell_m)}.
\end{equation} 
Note, that each cube $8P_m$ and the corresponding summand $\BB(8P_m)$ in (\ref{3.42}) occurs through a chaining 
originating from a descendant cube. Also, the quantities $\ell_m^{2n+p+\Ge(1-p)}\rho(\ell_m)$ and 
$\ell_{m+1}^{2n+p+\Ge(1-p)}\rho(\ell_{m+1})$ are comparable due to (C2) and the inverse doubling property (H$\rho$2).
Consequently, we can estimate 
\begin{align}
\label{3.43}
\sum_{j=1}^\infty I_j & \leq C\sum_{j=1}^\infty \ell_k^{2n+p+\Ge(1-p)}\rho(\ell_k)\sum_{m=0}^{k-1} \frac{\BB(8P_m)+\BB(8P_{m+1})}{\ell_m^{2n+p+\Ge(1-p)}\rho(\ell_m)}\\ \nonumber
&\leq C \sum_{Q\in \mathcal{W}} \BB(8Q)\sum_{P\in \mathcal{W}, P\leq Q}  
\frac{\ell_P^{2n+p+\Ge(1-p)}\rho(\ell_P)}{\ell_Q^{2n+p+\Ge(1-p)}\rho(\ell_Q)},
\end{align}
where the symbol $P\leq Q$ means $P$ is a descendant of $Q$ through the tree in Lemma~3.2. In the last step, 
following the lines (\ref{3.21}) (where we have a similar sum to estimate), we choose $\Ge=\frac{1}{2c_3(p-1)},$
that gives $\Ge(p-1)=\Gd$ as in (\ref{3.24}). It remains to estimate the descendant subtree sum
\begin{align*}
\sum_{P\in \mathcal{W}, P\leq Q} \ell_P^{2n+p-\Gd}\rho(\ell_P).
\end{align*}
This will be done by a geometric argument. Namely, we have the subtree inclusion 
$$
\mathrm{\bf{ST}}(Q)=\cup_{P\in \mathcal{W}, P\leq Q}P\subseteq KQ \subseteq B(x_Q,\sqrt{n}K\ell_Q),
$$ 
thus we can further estimate 
\begin{align}
\label{3.44}
\int_{\mathrm{\bf{ST}}(Q)}\int_{\mathrm{\bf{ST}}(Q)}|x-y|^{p-\Gd}\rho(x,y)dxdy
\leq \int_{B(x_Q,\sqrt{n}K\ell_Q)}\int_{B(x_Q,\sqrt{n}K\ell_Q)}|x-y|^{p-\Gd}\rho(x,y)dxdy,
\end{align}
Next, we have by the quasi-monotonicity of $\rho,$ that 
\begin{align*}
\int_{\mathrm{\bf{ST}}(Q)}\int_{\mathrm{\bf{ST}}(Q)}|x-y|^{p-\Gd}\rho(x,y)dxdy 
&\geq \sum_{P\in \mathcal{W}, P\leq Q}\int_{P}\int_{P}|x-y|^{p-\Gd}\rho(x,y)dxdy\\
&\geq  \sum_{P\in \mathcal{W}, P\leq Q}
\int_{B\left(x_P,\frac{\ell_P}{2}\right)} \int_{B\left(x_P,\frac{\ell_P}{2}\right)} |x-y|^{p-\Gd}\rho(x,y)dxdy\\
&\geq  \sum_{P\in \mathcal{W}, P\leq Q}
\rho(\ell_P)\int_{B\left(x_P,\frac{\ell_P}{2}\right)} \int_{B\left(x_P,\frac{\ell_P}{2}\right)} |x-y|^{p-\Gd}dxdy\\
&=C(n,p)\sum_{P\in \mathcal{W}, P\leq Q}\ell_P^{2n+p-\Gd}\rho(\ell_P),
\end{align*} 
for some constant $C(n,p)>0.$ We estimate the right-hand side of (\ref{3.44}) accordingly utilizing the radial property of 
$\rho$ and passing to polar coordinates and invoking (H$\rho$1')-(H$\rho$2') and the improved $\Gd-$margin moment bound (\ref{3.22}):
\begin{align*}
\int_{B(x_Q,\sqrt{n}K\ell_Q)}\int_{B(x_Q,\sqrt{n}K\ell_Q)} |x-y|^{p-\Gd}\rho(x,y)dxdy
&\leq \int_{B(x_Q,\sqrt{n}K\ell_Q)} \int_{B(0,2\sqrt{n}K\ell_Q)}|x|^{p-\Gd}\rho(x)dxdy\\
&=C(n)K^n\ell_Q^n \int_{B(0,2\sqrt{n}K\ell_Q)}|x|^{p-\Gd}\rho(x)dx\\
&=C(n)K^n\ell_Q^{n} \int_0^{2\sqrt{n}K\ell_Q} t^{n+p-1-\Gd}\rho(t)dt\\
&\leq C(n)K^nc_3\ell_Q^{2n+p-\Gd}\rho(2\sqrt{n}K \ell_Q)\\
&\leq C(n)K^nc_3\ell_Q^{2n+p-\Gd}\rho(\ell_Q).
\end{align*}
Finally, we can put together the estimates for both sides of (\ref{3.44}) to discover the desired bound 
\begin{align*}
\sum_{P\in \mathcal{W}, P\leq Q} \ell_P^{2n+p-\Gd}\rho(\ell_P)\leq  Cc_3\ell_Q^{2n+p-\Gd}\rho(\ell_Q),
\end{align*}
which together with (\ref{3.43}) gives 
\begin{align}
\label{3.45}
\sum_{j=1}^\infty I_j \leq Cc_3 \sum_{Q\in \mathcal{W}} \BB(8Q). 
\end{align}
It remains to employ the bounded overlap property of the cubes $8Q,$ $Q\in \mathcal{W},$ which is a consequence of 
(C1)-(C2). Theorem~2.3 is now follows from Theorem~2.4.

\end{proof}



\section{Truncated norm and seminorm equivalence}
\setcounter{equation}{0}
 \label{Sec:4}
 
We will estimate the full seminorm spreading the sum over products of Witney cubes and utilizing the same chaining idea as in [\ref{bib:Pra.Tos.},\ref{bib:Pra.Sak.}], see also [\ref{bib:Rutkowski.1}]. However, we will provide a slightly different in spirit proof, 
 by working with the norms directly rather than making use of duality. We start with recalling the main ideas of Whitney cube 
 chaining in uniform domains, and we refer the interested reader to the papers of Prats and Tosla [\ref{bib:Pra.Tos.}] and 
 Prats and Saksman [\ref{bib:Pra.Sak.}] for more details. Throughout this section, $\GO\subset\mathbb R^n$ is a bounded 
domain. We consider a standard Whitney decomposition of $\GO$ into a countable family of closed dyadic cubes 
$\mathcal{W}=\{Q_j\}_{j\in \mathcal{J}}$ satisfying the conditions (C1)-(C3) below; see [\ref{bib:Stein}] for decompositions of 
nonempty open subsets of $\mathbb R^n$.
\begin{itemize}
\item[(C1)] $C_{\mathcal{W}}\mathrm{diam}(Q)\leq \mathrm{dist}(Q,\partial\GO)\leq 4C_{\mathcal{W}}\mathrm{diam}(Q)$ for all $Q\in \mathcal{W}.$ 
\item[(C2)] $\frac{1}{2}\mathrm{diam}(Q)\leq \mathrm{diam}(S)\leq 2\mathrm{diam}(Q)$ if $Q\cap S\neq \emptyset.$
\item[(C3)] $\cup_{j\in \mathcal{J}}Q_j=\GO.$
\end{itemize}
The constant $C_{\mathcal{W}}>0$ can be chosen sufficiently large if needed, by further dividing the existing dyadic cubes in half in several steps. For two cubes, their long distance is defines as
\begin{equation}
\label{4.1}
D(Q,S):=\ell(Q)+\mathrm{dist}(Q,S)+\ell(S).
\end{equation}

\begin{definition}[Admissible chains]
\label{Def:4.1}
A chain joining distinct Whitney cubes $Q,S\in\mathcal W$ is a finite ordered family 
$$
[Q,S]:=(Q_1,\ldots,Q_M),\qquad Q_1=Q,\quad Q_M=S,
$$
of Whitney cubes, whose consecutive cubes are face-neighbors. Its length is defined as
$$
\ell([Q,S]):=\sum_{j=1}^M\ell(Q_j).
$$
For $\Ga>0$, the chain is $\Ga$-admissible if
\begin{equation}
\label{4.2}
\ell([Q,S])\le\frac{1}{\Ga}D(Q,S)
\end{equation}
and there is a central index $j_0$ such that 
\begin{equation}
\label{4.3}
\begin{aligned}
\ell(Q_j)&\ge\Ga D(Q,Q_j),&&1\le j\le j_0,\\
\ell(Q_j)&\ge\Ga D(Q_j,S),&&j_0\le j\le M.
\end{aligned}
\end{equation}
The cube $Q_S:=Q_{j_0}$ is called the central cube of the chain. Chains may be chosen without repeated cubes; 
a one-cube chain is allowed when $Q=S$.
\end{definition}

It is known that, uniform domains admit a Whitney covering $\mathcal W$ and a parameter $\Ga>0$ such that every pair $Q,S\in\mathcal W$ can be connected by an $\Ga$-admissible chain. This fact is well-known and widely used in harmonic analysis.

\begin{definition}[Shadows]
\label{Def:4.2}
For $P\in\mathcal W$ and $\tau>0$, define the collection
\begin{equation}
\label{4.4}
\operatorname{\bf{SH}}_{\tau}(P)
:=\{Q\in\mathcal W:Q\subset B(x_P,\tau\ell(P))\}
\end{equation}
and its so-called realization
\begin{equation}
\label{4.5}
\operatorname{\bf{Sh}}_{\tau}(P) :=\bigcup_{Q\in\operatorname{\bf{SH}}_{\tau}(P)}Q.
\end{equation}
There is a fixed $\tau_\Ga>1$, depending on the geometric constants of the domain, such that every $\Ga-$admissible 
chain $[Q,S]$ joining $Q$ and $S$ satisfies
$$
P\in[Q,Q_S]\ \Longrightarrow\ Q\in\operatorname{\bf{SH}}_{\tau_\Ga}(P),
\qquad
P\in[Q,S]\ \Longrightarrow\ P\in\operatorname{\bf{SH}}_{\tau_\Ga}(Q_S).
$$
Assuming that $\Ga$ and $\tau_\Ga$ are already fixed, we will abbreviate 
$\operatorname{\bf{SH}}:=\operatorname{\bf{SH}}_{\tau_\Ga}$ 
and $\operatorname{\bf{Sh}}:=\operatorname{\bf{Sh}}_{\tau_\Ga}$.
\end{definition}

We will be utilizing following lemma in the sequel. It is proven in [\ref{bib:Pra.Sak.}], see also the lines between 
equations (\ref{3.39})-(\ref{3.42}) for a proof of a similar estimate.

\begin{lemma}[{[\ref{bib:Pra.Sak.}, Remark~2.6]}]
\label{Lem:4.3}
For every $s>0$ and $Q\in\mathcal W$ one has
\begin{equation}
\label{4.6}
\sum_{L\in\mathcal W\ : \ Q\in\operatorname{SH}(L)}\ell(L)^{-s}
\le C_s\ell(Q)^{-s}.
\end{equation}
If $Q\in\operatorname{SH}(P)$ and $[Q,P]$ is an admissible chain, then
\begin{equation}
\label{4.7}
\sum_{L\in[Q,P]}\ell(L)^s\le C_s\ell(P)^s,
\qquad
\sum_{L\in[Q,P]}\ell(L)^{-s}\le C_s\ell(Q)^{-s}.
\end{equation}
In this situation one also has
\begin{equation}
\label{4.8}
D(L,P)\simeq\ell(P),\qquad L\in[Q,P].
\end{equation}
The constants depend on $s$ and the geometric constants of $\GO.$
\end{lemma}

We are now ready to prove Theorem~2.4
\begin{proof}[Proof of Theorem~2.4] 

In what follows, the constants $C_i,C,c>0$ may depend on the dimension, the exponent $p,$ the constants of the kernel $\rho$
 in the hypotheses utilized, the Whitney covering constants, and the uniformity parameters. For an integrable function $f$, 
we write $f_Q=\fint_Q f(x) dx$. We initially make the constant $C_{\mathcal W}$ in (C1) so big that $x\in Q$ implies 
$5Q\subset B(x,\Gth\delta(x)).$ We start with calculating the "larger" part of the truncated seminorm:
\begin{align}
\label{4.9}
\int_{\GO} \int_{\GO \setminus B(x,\Gth\delta(x))}\rho(x,y)|f(x)-f(y)|^pdydx
&\leq \sum_{j=0}^\infty \int_{Q_j}\int_{\GO \setminus 5Q_j}\rho(x,y)|f(x)-f(y)|^pdydx\\   \nonumber
&\leq C \sum_{Q_i\cap Q_j=\emptyset }\int_{Q_i}\int_{Q_j}\rho(x,y)|f(x)-f(y)|^pdydx\\ \nonumber
&\leq C \sum_{Q_i\cap Q_j=\emptyset }\int_{Q_i}\int_{Q_j} \rho(x,y)|f(x)-f_{Q_i}|^pdydx\\ \nonumber
&+C  \sum_{Q_i\cap Q_j=\emptyset }\int_{Q_i}\int_{Q_j}\rho(x,y)|f(y)-f_{Q_j}|^pdydx\\ \nonumber
&+C  \sum_{Q_i\cap Q_j=\emptyset }\int_{Q_i}\int_{Q_j}\rho(x,y)|f_{Q_i}-f_{Q_j}|^pdydx\\ \nonumber
&:=C(S_1+S_2+S_3).
\end{align} 
We can estimate the summands $S_1$ and $S_2$ by the "smaller" part of the truncated seminorm in the exact same way by symmetry. Denote $D=\mathrm{diam}(\GO).$ Note first, that in fact hypothesis (H$\rho$5) splits into two parts. Namely, if we assume 
$\lim_{a\to 0+}\int_a^1t^{n-1}\rho(t)dt=c_0<\infty,$ this will imply that $\int_0^Dt^{n-1}\rho(t)td=C_0<\infty,$ and thus we can simply estimate passing to polar coordinates and by Minkowski:
\begin{align*}
\int_{\GO} \int_{\GO}\rho(x,y)|f(x)-f(y)|^pdydx&\leq2^{p-1} \int_{\GO} \int_{\GO}\rho(x,y)\left(|f(x)|^p+|f(y)|^p\right)dydx\\
&\leq C \int_{\GO} \int_0^Dt^{n-1}\rho(t)dt |f(x)|^pdx\\
&=C\int_{\GO}|f(x)|^pdx,
\end{align*}
and we are done. If we assume otherwise $\lim_{a\to 0+}\int_a^1t^{n-1}\rho(t)dt=\infty,$ then we can choose $a_0\in (0,1)$ such that 
$\int_{a_0}^1t^{n-1}\rho(t)dt>2c_5,$ thus for $a\in (0,a_0]$ hypothesis (H$\rho$5) implies 
$$
\int_{a}^1t^{n-1}\rho(t)dt\leq 2c_5a^n\rho(a),
$$
which is (H$\rho$4) with a doubled constant and validity interval $(0,a_0).$ But we can increase the validity interval to $(0,1)$ due to 
quasi-monotonicity of $\rho$ with possibly a larger constant $c_5.$ This being said, we will work under (H$\rho$4) in the rest of the proof. Another remark is, that (H$\rho$4) implies 
$$
\int_{a}^Dt^{n-1}\rho(t)dt\leq Ca^n\rho(a),\qquad a\in (0,D)
$$
again due to the quasi-monotonicity of $\rho.$ Now, owing to (C2), we have $\mathrm{dist}(Q_i,Q_j)\geq \frac{1}{2}\max(\ell_i,\ell_j)$ 
as long as $Q_i\cap Q_j=\emptyset,$ thus we bound $S_2$ passing to polar coordinates and utilizing (H$\rho$1'), (H$\rho$2') 
and (H$\rho$4) (or (H$\rho$5)) as follows:
\begin{align}
\label{4.10}
S_2&\leq \sum_{Q_i\cap Q_j=\emptyset }\int_{Q_i}\int_{Q_j}\rho(x,y)|f(y)-f_{Q_j}|^pdydx\\ \nonumber
&\leq C\sum_j \int_{\frac{\ell_j}{2}}^Dt^{n-1}\rho(t)\int_{Q_j}|f(y)-f_{Q_j}|^pdy\\ \nonumber
&\leq C\sum_j \int_{Q_j}\ell_j^n\rho(\ell_j)|f(y)-f_{Q_j}|^pdy.
\end{align} 
Next, we have by Jensen's inequality, that
$$
\int_{Q_j} |f(y)-f_{Q_j}|^pdy \leq \frac{1}{|Q_j|}\int_{Q_j}\int_{Q_j}|f(y)-f(x)|^pdydx,
$$
thus we get from (\ref{4.10}) by Minkowsi and by (H$\rho$1') and (H$\rho$2'), the desired bound
\begin{equation}
\label{4.11}
S_2 \leq C\sum_j \int_{Q_j}\int_{Q_j}\rho(x,y)|f(y)-f(x)|^pdydx.
\end{equation}
Now in order to estimate $S_3,$ we need to connect $Q_i$ and $Q_j$ through admissible chains of cubes. Consequently, 
let $[Q_i,Q_j]=(P_1,\dots,P_M)$ be an admissible chain of Whitney cubes connecting $Q_i$ and $Q_j$ and denote 
$d_{ij}:=\mathrm{dist}(Q_i,Q_j)$ for notational convenience. We can write by the triangle inequality as in the lines between 
equations (\ref{3.27}) and (\ref{3.28}):
$$
|f_{Q_i}-f_{Q_j}|\leq\sum_{m=1}^{M-1}|f_{P_{m}}-f_{P_{m+1}}|,
$$
thus we estimate $S_3$ by (\ref{4.2}), (\ref{4.3}), (H$\rho$1'), and (H$\rho$2') as follows:
 \begin{align}
\label{4.12}
S_3&=\sum_{Q_i\cap Q_j=\emptyset}\int_{Q_i}\int_{Q_j}\rho(x,y)|f_{Q_i}-f_{Q_j}|^pdydx\\ \nonumber
&\leq C\sum_{Q_i\cap Q_j=\emptyset}\ell_i^n\ell_j^n\rho(d_{ij})\left(\sum_{m=1}^{M-1}|f_{P_{m}}-f_{P_{m+1}}|\right)^p\\ \nonumber
&:=C\sum_{Q_i\cap Q_j=\emptyset}S_3^{ij}
\end{align}   
We will apply H\"oder's inequality to $S_3^{ij}$ for each pair $(i,j)$ with the right weights like in
 (\ref{3.37})-(\ref{3.38}). To that end, let $\mu_{ij}(m)>0$ be positive weights yet to be chosen such that 
 \begin{align}
\label{4.13}
\sum_{m=1}^{M-1}\mu_{ij}(m) \leq C\qquad\text{uniformly in}\qquad (i,j). 
\end{align}   
Then we can write for $p>1$ by H\"older with $q=\frac{p}{p-1}$ (for $p=1$ no application of H\"older's inequality is needed):
\begin{equation}
\label{4.14}
\left(\sum_{m=1}^{M-1}|f_{P_{m}}-f_{P_{m+1}}|\right)^p
\leq \left(\sum_{m=1}^{M-1}\mu_{ij}(m) \right)^{p-1}\sum_{m=1}^{M-1}\frac{1}{\mu_{ij}(m) ^{p-1}}|f_{P_{m}}-f_{P_{m+1}}|^p.
\end{equation}
We have again by Jensen, that 
$$
|f_{P_{m}}-f_{P_{m+1}}|^p\leq \frac{C}{\ell_{P_m}^n\ell_{P_{m+1}}^n} \int_{P_{m}}\int_{P_{m+1}} |f(x)-f(y)|^pdxdy,
$$
thus, taking into account the fact that $\frac{1}{2}\ell_{P_m}\leq \ell_{P_{m+1}}\leq 2\ell_{P_m},$ we 
have by (\ref{4.13})-(\ref{4.14}) and by (H$\rho$1') and (H$\rho$2'), the bound
 \begin{align}
\label{4.15}
S_3^{ij}&\leq C\sum_{m=1}^{M-1}\frac{\rho(d_{ij})}{\rho(\ell_{P_m})}\cdot \frac{\ell_i^n\ell_j^n }{\mu_{ij}(m)^{p-1}\ell_{P_m}^{2n}}
 \int_{P_{m}}\int_{P_{m+1}}\rho(x,y) |f(x)-f(y)|^pdxdy,
\end{align}   
Let now $\Ge>0$ be a small constant yet to be chosen and choose 
\begin{equation}
\label{4.16}
\mu_{ij}(m)=\frac{\ell_{P_m}^\Ge}{\ell_{P_{m_0}}^\Ge},\qquad m=1,\dots, M-1,
\end{equation}
where $m_0$ is the central index in (\ref{4.3}). Definition~4.2 and Lemma~4.3 then insure the uniform bound 
$\sum_{m=1}^{M-1}\mu_{ij}(m)\leq C$ in (\ref{4.13}). It remains to estimate the coefficient of the integral over 
$P_{m}\times P_{m+1}$ in (\ref{4.15}) in the total summation. We will replace the integration set 
$P_{m}\times P_{m+1}$ by the $P_{m}\times 5P_m$ due to the inclusion  $P_{m+1}\subset 5P_m,$ to make it depend 
on a single cube for summation ease purposes in the sequel. Denoting for notational brevity 
$$
W_P=\int_{P}\int_{5P}\rho(x,y) |f(x)-f(y)|^pdxdy,
$$
we can estimate by (\ref{4.15}) and (\ref{4.16}):
 \begin{align}
\label{4.17}
S_3&\leq C\sum_{Q_i\cap Q_j=\emptyset}S_3^{ij}\\ \nonumber
&\leq C\sum_{P\in\mathcal W} W_p \sum_{Q\cap S=\emptyset,\ P\in [Q,S]}
 \frac{\ell_Q^n\ell_S^n\ell_{Q_S}^{\Ge(p-1)}\rho(D(Q,S))} {\ell_P^{2n+\Ge(p-1)}\rho(\ell_P)},
\end{align}   
thus the estimation of $S_3$ reduces to proving the bound 
\begin{equation}
\label{4.18}
\sum_{Q\cap S=\emptyset,\ P\in [Q,S]} \ell_Q^n\ell_S^n \ell_{Q_S}^{\Ge(p-1)}\rho(D(Q,S))
\leq C\ell_P^{2n+\Ge(p-1)}\rho(\ell_P),\qquad \forall P\in\mathcal W,
\end{equation}
We will prove (\ref{4.18}) by a geometric argument like in (\ref{3.43})-(\ref{3.45}) as follows. Set $\eta=\Ge(p-1)$ and consider the double integral 
$$
\int_{\mathrm{\textbf{Sh}}(P)}\int_{|x-y|\geq \beta \ell_P}|x-y|^\eta\rho(x,y),
$$
where $\beta>0$ is a small constant such that for every Whitney cube $Q\in\mathcal W$ and any admissible chain $[Q,S]$ that contains $P,$ the set $\{y\in\GO \ : \ x\in Q,\ |x-y|\geq  \beta \ell_P\}$ contains the terminal cube $S.$ This is possible by 
the properties of admissible chains. Now, on one hand we can estimate the integral from below by the shadow inclusions in Definition~4.2 and utilizing (H$\rho$1') and (H$\rho$2') as follows:
 \begin{align}
\label{4.19}
\int_{\mathrm{\textbf{Sh}}(P)}\int_{|x-y|\geq \beta \ell_P}|x-y|^\eta\rho(x,y)dydx&\geq 
\sum_{Q\in \mathrm{\textbf{SH}}(P), \ Q\cap S=\emptyset,\ P\in [Q,Q_S]}\int_Q\int_{S} |x-y|^\eta\rho(x-y)dydx\\ \nonumber
&\geq C\sum_{Q\cap S=\emptyset,\ P\in [Q,Q_S]} \ell_Q^n\ell_{S}^nD(Q,S)^\eta\rho(D(Q,Q_S))\\ \nonumber
&\geq C\sum_{Q\cap S=\emptyset,\ P\in [Q,Q_S]} \ell_Q^n\ell_{S}^n\ell_{Q_S}^\eta\rho(D(Q,Q_S)).
\end{align}   
Before estimating the integral from above, we need a $\eta-$margin version of (H$\rho$4), namely, for $\eta=\frac{1}{2c_4},$ the 
improved moment bound holds
\begin{equation}
\label{4.20}
\int_a^1 t^{n-1+\eta}\rho(t)\leq 2c_4a^{n+\eta}\rho(a),\qquad 0<a<1.
\end{equation}
This can be proven by a straightforward bootstrapping argument as in the proof of (\ref{3.22}) and we leave it to the reader. Now the choice of $\eta$ gives $\Ge=\frac{1}{2c_4(p-1)}.$ Consequently, we can estimate the integral from above as
 \begin{align}
\label{4.21}
\int_{\mathrm{\textbf{Sh}}(P)}\int_{|x-y|\geq \beta \ell_P}|x-y|^\eta\rho(x,y)dydx
&\leq  \int_{B(x_P,\tau\ell_P)}\int_{\beta\ell_P}^D t^{n-1+\eta}\rho(t)\\ \nonumber
&\leq C\ell_P^{2n+\eta}\rho(\ell_P).
\end{align}   
Now, (\ref{4.19}) and (\ref{4.21}) yield (\ref{4.18}). Summing up, we have proven the bound 
\begin{align}
\label{4.22}
\int_{\GO} \int_{\GO \setminus B(x,\Gth\delta(x))}\rho(x,y)|f(x)-f(y)|^pdydx
&\leq C\sum_{j=0}^\infty \int_{Q_j}\int_{5Q_j}\rho(x,y)|f(x)-f(y)|^pdydx\\   \nonumber
&\leq C\int_{\GO} \int_{B(x,\Gth\delta(x))}\rho(x,y)|f(x)-f(y)|^pdydx,
\end{align} 
under the hypothesis (H$\rho$4), where by enlarging the constant $C_{\mathcal W}$ if necessary, we have 
$5Q\subset \{y \  : \ x\in Q, \ |x-y|< \Gth\delta(x)\}.$ This completes the proof of the theorem. 

\end{proof}


\section{Applications to rigidity estimates in nonlocal continuum mechanics}
\setcounter{equation}{0}
\label{Sec:5}

As mentioned in Section~1, the next two examples are fundamental in nonlocal continuum mechanics (Peridynamics). 
In all four examples below, we take $m_1=m_2=N=n.$\\

\noindent \textbf{The FKI.} In this case the operators $\BA$ and $\BB$ are defined as follows:
\begin{equation}
\label{5.1}
\BA(Q,f)(x)=f(x)-\Ba_Q(x),
\end{equation}
and 
\begin{equation}
\label{5.2}
\BB(\GO,f)(x,y)=\frac{(f(x)-f(y))\cdot(x-y)}{|x-y|},
\end{equation}
where $\Ba_Q=\BT x+b\in\mathfrak{se}(n) $ is the minimizer in (\ref{1.5}) on the cube $Q,$ that exists due to the fact that 
the space $\mathfrak{se}(n)$ is finite dimensional. Note, that indeed $\BA$ and $\BB$ obviously satisfy (\ref{1.2})-(\ref{1.3}) and 
$\BA$ fulfills hypotheses (HA1) and (HA2) due to the fact that $\BA(Q,f)(x)-\BA(P,f)(x)$ is affine for any two cubes $P$ and $Q.$\\

\noindent \textbf{The FGRI.} In this case the operators $\BA$ and $\BB$ are defined in a similar way:
\begin{equation}
\label{5.3}
\BA(Q,f)(x)=f(x)-\Br_Q(x),
\end{equation}
and 
\begin{equation}
\label{5.4}
\BB(\GO,f)(x,y)=|f(x)-f(y)|-|x-y|,
\end{equation}
where $\Br_Q=\BR x+b$ is the minimizer in (\ref{1.10}) on the cube $Q,$ that exists due to the fact that  $\mathfrak{so}(n)$ is finite dimensional. It is straightforward to check that both $\BA$ and $\BB$ fulfill the conditions (\ref{1.2}), (\ref{1.3}) and hypotheses in 
(HA1) and (HA2) for the exact same reason as in the case of FKI.\\

The reduced inequality (\ref{2.1}) is proven in the companion paper [\ref{bib:Har.Mik.2}, Theorems~2.5 and 2.6] for both of the above 
pairs of operators , thus yielding Theorem~2.3 for both FKI and FGRE. Next, we consider two new rigidity implications in nonlocal 
continuum mechanics. \\

\noindent \textbf{Rigidity of homogeneous deformations.} We would like to derive an analogous to the above two situations rigidity result for homogeneous, i.e., affine deformations. Affine deformations have the property, that they are both convex and concave, i.e., satisfy the midpoint rule 
$$
f(x)+f(y)=2f\left(\frac{x+y}{2}\right).
$$

In this case, the operators $\BA$ and $\BB$ are defined in a similar way as follows. For $f\in L^1_{loc}(\GO,\mathbb R^n),$ consider minimization problem 
\begin{equation}
\label{5.5}
\min_{\BT\in\mathbb M^{n\times n}, \  b\in\mathbb R^n} \int_{Q}|f(x)-\BT x-b|^pdx.
\end{equation}
Define 
\begin{equation}
\label{5.6}
\BA(Q,f)(x)=f(x)-\BT_Q x-b_Q,
\end{equation}
to be the minimizer in (\ref{5.5}) and for $x,y\in\GO,$ define $\BB$ as
\begin{equation}
\label{5.7}
\BB(\GO,f)(x,y)=
\begin{cases}
f(x)+f(y)-2f\left(\frac{x+y}{2}\right) \  &  \quad \text{if}\quad \frac{x+y}{2}\in \GO,\\
0                       \  &  \quad \text{if}\quad \frac{x+y}{2}\notin \GO,\\
\end{cases}
\end{equation}
The unweighted inequality 
\begin{equation}
\label{5.8}
\min_{\BT\in\mathbb M^{n\times N}, \  b\in\mathbb R^N} \int_{Q}|f(x)-\BT x-b|^pdx\leq 
\frac{1}{Q}\int_Q \int_Q \left|f(x)+f(y)-2f\left(\frac{x+y}{2}\right) \right|^pdxdy
\end{equation}
is proven in [\ref{bib:Har.Mik.2}, Theorem~2.7], thus the weighted nonlocal analogue follows: 
For some cube $Q\subset\GO,$ one has
\begin{align}
\label{5.9}
\int_\GO \int_{\GO}\rho(x,y)&|f(x)-f(y)-(\BT(x)-\BT(y))|^pdx\\ \nonumber
& \leq \int_\GO \int_{\left\{y \in \GO \ : \ \frac{x+y}{2}\in \GO \right\}} \rho(x,y)\left|f(x)+f(y)-2f\left(\frac{x+y}{2}\right)\right|^pdydx,
\end{align}
for all $f\in L^1_{loc}(\GO,\mathbb R^n),$ with $\BT(x)=\BT_Q x-b_Q.$ Inequality (\ref{5.9}) will be useful in second 
gradient models.\\

\noindent \textbf{Rigidity of dilations.} Next we derive an analogous rigidity result for dilations. The relative displacement in dilational
deformations is always parallel to the bond, thus one has 
$$
\left(\BI-\frac{x-y}{|x-y|}\otimes \frac{x-y}{|x-y|}\right)(f(x)-f(y))=0,\qquad\text{for all}\qquad x,y\in\GO.
$$
This observation suggest the definitions 
\begin{equation}
\label{5.10}
\BA(Q,f)(x)=f(x)-\Gl x-b_Q,
\end{equation}
where the mapping $\BGL(x)=\Gl_Q x-b_Q\colon Q\to\mathbb R^n$ with $\Gl_Q\in\mathbb R$ and $b_Q\in \mathbb R^n,$ is the minimizer in 
\begin{equation}
\label{5.11}
\min_{\Gl \in\mathbb R, \ b\in\mathbb R^n} \int_{Q}|f(x)-\Gl x-b|^pdx,
\end{equation}
and 
\begin{equation}
\label{5.12}
\BB(\GO,f)(x,y)=\left(\BI-\frac{x-y}{|x-y|}\otimes \frac{x-y}{|x-y|}\right)(f(x)-f(y)),\qquad\text{for all}\qquad x,y\in\GO.
\end{equation}
Analogously, the unweighted inequality 
\begin{equation}
\label{5.13}
\min_{\Gl \in\mathbb R, \ b\in\mathbb R^n} \int_{Q}|f(x)-\Gl x-b|^pdx\leq 
\frac{1}{Q}\int_Q \int_Q \left|\left(\BI-\frac{x-y}{|x-y|}\otimes \frac{x-y}{|x-y|}\right)(f(x)-f(y)) \right|^pdxdy
\end{equation}
is proven in [\ref{bib:Har.Mik.2}, Theorem~2.8], thus the weighted nonlocal analogue follows: 
For some cube $Q\subset\GO,$ one has
\begin{align}
\label{5.14}
\int_\GO \int_{\GO}\rho(x,y)&|f(x)-f(y)-(\BGL(x)-\BGL(y))|^pdx\\ \nonumber
& \leq \int_\GO \int_\GO \rho(x,y)\left|\left(\BI-\frac{x-y}{|x-y|}\otimes \frac{x-y}{|x-y|}\right)(f(x)-f(y)) \right|^dydx,
\end{align}
for all $f\in L^1_{loc}(\GO,\mathbb R^n),$ with $\BGL(x)=\Gl_Q x-b_Q.$


\section*{Acknowledgements}
This material is partially supported by the National Science Foundation under Grants No. DMS-2206239. 
We would like to thank Tadele Mengesha, Hayk Mikayelyan, and James Scott for many interesting discussions on the subject. 
We used ChatGPT-6 Astra for reference search and minimal typo and grammar editing to improve the presentation of the 
manuscript. The counting function $N_j$ in (\ref{3.11}) and the short bootstrapping proof for (\ref{3.22}) were suggested by 
the AI to shorten the Tonelli change of summation order argument, rather than counting combinatorially, and as an alternative 
very short proof of (\ref{3.22}), which we have a significantly longer proof (calculation) for. All of the other ideas, formulations,
and proofs are generated by us and the manuscript is fully written by us.


\end{document}